\documentclass[12pt,reqno]{amsart}
     
\usepackage[left=1in, right=1in, top=1.1in,bottom=1.1in]{geometry}
\usepackage[dvipsnames]{xcolor}
\usepackage[mathscr]{eucal}
\usepackage{latexsym}
\usepackage{amsfonts}
\usepackage{amssymb}
\usepackage{bm}
\usepackage{pifont}
\usepackage{stmaryrd}
\usepackage{graphicx}
\usepackage{color}
\usepackage{enumitem}
\usepackage{hyperref}
\hypersetup{
     colorlinks   = true,
     citecolor    = blue,
     linkcolor    = blue
} 

\usepackage{pdfsync}
\usepackage[font={scriptsize}]{caption}

\usepackage{tikz}
\usetikzlibrary{trees}

\definecolor{darkblue}{rgb}{0.1, 0.2, 0.75}
\definecolor{darkgreen}{rgb}{0.1, 0.35, 0}
\definecolor{candypink}{rgb}{0.89, 0.44, 0.48}
\definecolor{deepcerise}{rgb}{0.85, 0.2, 0.53}
\definecolor{atomictangerine}{rgb}{1.0, 0.6, 0.4}
\definecolor{fandango}{rgb}{0.71, 0.2, 0.54}

\newtheorem{theorem}{Theorem}[section]
\newtheorem{Def}[theorem]{Definition}

\newtheorem{thm}[theorem]{Theorem}

\newtheorem{prop}[theorem]{Proposition}

\newtheorem{lemma}[theorem]{Lemma}
\newtheorem{example}[theorem]{Example}
\theoremstyle{remark}
\newtheorem{remark}[theorem]{Remark}

\numberwithin{equation}{section}

\def\RR{\mathbb{R}}

\def\NN{\mathbb{N}}
\def\mE{\mathbb{E}}

\def\ZZ{\mathbb{Z}}

\def\bfy{{\bf y}}

\newcommand{\cc}{{\mathcal C}}
\newcommand{\cd}{{\mathcal D}}

\newcommand{\cf}{{\mathcal F}}

\newcommand{\ch}{{\mathcal H}}

\newcommand{\cj}{{\mathcal J}}

\newcommand{\cl}{{\mathcal L}}

\newcommand{\cs}{{\mathcal S}}

\def\si{\sigma}

\def\al{{\alpha}}

\def\be{{\beta}}

\def\ga{{\gamma}}

 \newcommand{\ep}{\varepsilon}

\newcommand{\lp}{\left(}
\newcommand{\rp}{\right)}

\def\ll{\llbracket}
\def\rr{\rrbracket}

\begin{document}
\title[variations of stochastic processes]
{Power    variations and limit theorems for     stochastic processes controlled  by fractional Brownian motions}
\date{}   

\author[Y. Liu]
{Yanghui Liu} 
\address{Y. Liu: Baruch College, CUNY, New York}
\email{yanghui.liu@baruch.cuny.edu}

\author[X. Wang]
{Xiaohua Wang} 
\address{X. Wang: New York City College of Technology, CUNY, New York}
\email{xiwang@citytech.cuny.edu}

\keywords{Power variation, discrete rough integral,  fractional Brownian motion, controlled rough path, limit theorems, estimation of volatility. }

\begin{abstract} 
In this paper we establish   limit theorems for   power variations of  stochastic processes controlled by fractional Brownian motions with Hurst parameter $H\leq 1/2$. We show that   the power variations of such processes can be decomposed into the mix of several weighted random sums plus some remainder terms, and the convergences of   power variations  are dominated by different combinations of those weighted sums depending on whether $H<1/4$, $H=1/4$, or  $H>1/4$.  We show that when   $H\geq 1/4$ the centered power  variation converges stably at the rates $n^{-1/2}$, and when $H<1/4$ it converges  in probability at  the rate    $n^{-2H}$.
We determine the limit of the mixed weighted sum based on a  rough path approach  developed in \cite{LT20}.
\end{abstract}

\maketitle

{
\hypersetup{linkcolor=black}
\setcounter{tocdepth}{1}
}

\section{Introduction}
 
 In this paper we establish limit theorems for power variations of  low-regularity processes in a general  rough path framework.  
Recall that for a  stochastic process $(y_{t}, t\in[0,1])$ the power variation of order $p>0$ ($p$-variation for short) is defined as
\begin{eqnarray}\label{e.power}
\sum_{k=0}^{n-1} \big|y_{t_{k+1}} - y_{t_{k}} \big|^{p}, 
\end{eqnarray}
where $0=t_{0}<t_{1}<\cdots<t_{n}=1$ is a partition of the time interval $[0,1]$. 
The power variation has been widely used in quantitative finance for  the estimation of  volatility and related  parameters; see  \cite{AJ, barndorff2002econometric, barndorff2003realized, barndorff2004econometric, barndorff2004power, chong2022statistical} and  references therein. 

When $y$ is a semimartingale the    power variation has been  discussed in  
\cite{barndorff2006limit, barndorff2003realized, barndorff2006, jacod08, kabanov2006central, kinnebrock2008note, woerner2003variational, woerner2005estimation}.  The case of  stationary Gaussian   was treated  in  \cite{guyon1989convergence, leon2000weak}.  
When $y $ is a Young integral (see \cite{young1936inequality}) driven by  fractional Gaussian processes  the power variation   has been  investigated   in  \cite{barndorff2009power, corcuera2006power, norvaivsa2015weighted}. The study of power variations \eqref{e.power} in the non-semimartingale case is closely related to the limits   of weighted random sums. For example, a key step in \cite{barndorff2009power, corcuera2006power, norvaivsa2015weighted} is to observe that  when $y$ is a Young integral of the form: $y_{t} = \int_{0}^{t} z_{s}dx_{s}$   and      the integrator $x$   is H\"older continuous of order greater than $1/2$, the increment $y_{t_{k+1}} - y_{t_{k}}$ in \eqref{e.power} can be replaced by its first-order approximation  $   z_{t_{k}}(x_{t_{k+1}} - x_{t_{k}})$.      
   We refer the  reader  to \cite{binotto2018weak, burdzy2010change, corcuera2014asymptotics, LT20, liu2023convergence, nourdin2008asymptotic, nourdin2016quantitative, nourdin2010central, nourdin2009asymptotic, nualart2019asymptotic} for discussions about limit theorems of weighted random sums.  

Recently,  empirical evidence was found that security  volatility  actually has much lower regularity  than semimartingales  (see \cite{GJR}). The statement is further  supported by other empirical work based on both return data (see \cite{bennedsen2016decoupling, FTW21, gatheral2020quadratic}) and option data (see \cite{BFG, F21, LMPR18}).  Motivated by these advances in quantitative finance, it is then   natural to ask the following question: Is there a limit theorem for power variations when the process is ``rougher'' than semimartingale, and if so, under what conditions does the limit theorem hold? 

A main difficulty in the low-regularity case is that the aforementioned relation between $y$ and its first-order approximation  is no longer true. In fact, we will see that  the difference between the power variations of  a low-regularity process $y$ and that of its first-order approximation has nonzero contribution  to the limit of   power variation. A second difficulty is that the weighted sums corresponding to \eqref{e.power} involve   functionals of the forms $|x|^{p}$ and $|x|^{p}\cdot \text{sign}(x)$, where $x$ is the underlying Gaussian process for $y$ (see Definition \ref{def.control} for our definition of the processes of $x$ and $y$). Both functionals are in the  infinite sum of chaos when $p$ is non-integer and so a direct approach of Malliavin integration by parts for the weighted sums is not possible (Recall that the integration by parts is a crucial step in the study of weighted sums in \cite{binotto2018weak, nourdin2008asymptotic, nourdin2016quantitative, nourdin2010central, nourdin2009asymptotic, nualart2019asymptotic}).  

  
In this paper we show that a limit theorem of power variation does hold under the assumption  that  $y$ is a process ``controlled'' by a fractional Brownian motion (fBm for short). The controlled process is a main concept in the theory of rough path,  and it is broad enough to contain two important  models of stochastic  processes   we have in mind: the  rough integrals and the rough differential equations (see Example \ref{example.control}). Our result generalizes \cite{corcuera2006power} to fBms with any Hurst parameter $H\in(0,1)$.

 
 
 

 
 
 
Our main result can be informally stated as follows. 
The reader is  referred  to Theorem \ref{thm.main}  for a more precise statement.

 \begin{thm}
Let $x$  be 
       a fBm with Hurst parameter $H\leq 1/2$ and   $(y ,y',\dots, y^{(\ell-1)})$ be a   process controlled by $x$ almost surely (see Definition \ref{def.control}) for some $\ell\in\NN$. Define the function $\phi(x)=|x|^{p}$, $x\in\RR$ for some constant $p>0$, and denote    $\phi'$ and $\phi''$  the derivatives  of $\phi$, and   $c_{p}$ and $\si$ are     constants   given  in \eqref{e.cp} and \eqref{e.pvar2}, respectively.   
Let 
\begin{eqnarray*}
U^{n}_{t} = n^{pH-1}
\sum_{0\leq t_{k}<t} \phi(y_{t_{k+1}} - y_{t_{k}} )  - c_{p}\int_{0}^{t} \phi(y'_{u}) du
 \qquad t\in[0,1].  
\end{eqnarray*}
 Let   $W$ be a standard Brownian motion independent of $x$.  Then 

\noindent (i) For $1/4< H\leq1/2$, $p\in[3,\infty)\cup\{2\}$ and $\ell\geq4$  we have the convergence in law:  
\begin{eqnarray*}
n^{1/2}U^{n}_{1}   \to \si \int_{0}^{1} \phi(y_{u}') dW_{u} ,
\qquad \text{as $n\to\infty$. }
\end{eqnarray*}

\noindent (ii) For $H=1/4$, $p\in[5,\infty)\cup\{2,4\}$ and $\ell\geq 6$    the following  convergence in law holds:
\begin{eqnarray*}
  n^{1/2}U^{n}_{1}  \to \si \int_{0}^{1} \phi(y_{u}')dW_{u} 
- \frac{c_{p}}{8} \int_{0}^{1}  
    \phi''(y_{u}')(y_{u}'')^{2} du+ \frac{(p-2)c_{p}}{24} \int_{0}^{1}  \phi'(y_{u}') y_{u}''' 
    du 
\end{eqnarray*}
as $n\to\infty$. 


 \noindent (iii) For $H<1/4$, $p\in[5,\infty)\cup\{2,4\}$  and $\ell\geq 6$ 
 we have the  convergence in probability: 
\begin{eqnarray*}
 n^{2H}U^{n}_{1}   \to - \frac{c_{p}}{8} \int_{0}^{1}  
    \phi''(y_{u}')(y_{u}'')^{2} du+ \frac{(p-2)c_{p}}{24} \int_{0}^{1}  \phi'(y_{u}') y_{u}''' 
    du
\qquad
  \text{as $n\to\infty$.  }
\end{eqnarray*}

\end{thm}

 
 As  mentioned previously,   the limit of   power variation in the low-regularity case is not solely determined by the first-order approximation of   $y$. A first step of our proof is thus to consider the higher-order approximation of $y$ and to    estimate  the corresponding weighted random sums and  remainder terms.  
  The convergences of      mixed weighted sums and   power variation are  based on a rough path approach developed in \cite{LT20}.  In particular, we will see that   the rough path approach  allows us to avoid the application of Malliavin integration by parts for functionals of infinite chaos.  


 
 
 The  paper is structured as follows: In Section \ref{section.rp} we introduce the concept of discrete rough paths and discrete rough integrals and recall some basic results of the rough paths theory. In Section \ref{section.bd} we derive some useful estimates and limit theorem results for weighted random sums related to fBm. In Section \ref{section.main} we prove the limit theorem of power variation for processes controlled by fBm. 
 

\subsection{Notation}\label{section.notation} 
Let $\pi:0=t_{0}<t_{1}<\cdots<t_{n}=1$ be a partition on $[0,1]$. For $s,t\in [0,1]$ such that $s<t$,  we write $\ll s,t \rr$ for the discrete interval that consists of   $t_{k} $'s such that   $ t_{k}\in [s,t] $ and   the two endpoints $s$ and $t$. Namely, $\ll s, t\rr = \{t_{k}: s\leq t_{k} \leq t\}\cup \{s,t\}$.  For $N\in\NN=\{1,2,\dots\}$ we denote   the discrete simplex  $\cs_{N}(\ll s,t\rr)=
\{ (u_{1},\dots, u_{N}) \in \ll s,t\rr^{N} :\, u_{1}< \cdots< u_{N}  \}$. Similarly, we denote the continuous simplex: $\cs_{N}([s,t])= \{ (u_{1}  ,\dots, u_{N} ) \in [ s,t]^{N} :\, u_{1} < \cdots< u_{N}  \}$.  

Throughout the paper we work on a  probability space $(\Omega, \mathscr{F}, P)$. If $X$ is a random variable, we denote by $| X |_{L_{p}}  $ the $L_{p} $-norm of $X$.
The letter $K$  stands  for  a constant independent of    any important parameters which can change from line to line. We write $A\lesssim B$ if there is a constant $K>0$ such that $A\leq KB$. We denote $[a]  $   the integer part of   $a$. 

\section{Preliminary results} \label{section.rp}
In this section, we introduce the concept of discrete rough paths and discrete rough integrals, and recall some basic results of the rough paths theory. In the second part of the section we recall the elements  of Wiener chaos expansion and   fractional Brownian motion. 

\subsection{Controlled rough paths and algebraic properties}
This subsection is devoted to introducing the main rough paths notations which will be used in the sequel.   The reader is referred to   \cite{FH20, FV10} for an introduction to the rough path theory.

Recall that the continuous simplex $\cs_{k}([0,1])$ is defined in Section \ref{section.notation}. 
We denote by $\cc_{k} $ the set of functions $g : \cs_{k}([0,1]) \to \RR$ such that $g_{u_{1}\cdots u_{k}} = 0$ whenever $u_{i} = u_{i+1}$ for some $i\leq k -1$. Such a function will be called a $(k-1)$-\emph{increment}. We define the operator $\delta$ as follows:
\begin{eqnarray*}
\delta : \cc_{k}  \to \cc_{k+1} , \quad\quad (\delta g)_{u_{1}\cdots u_{k+1}} = \sum_{ i=1}^{k+1} (-1)^{i} g_{u_{1}\cdots \hat{u}_{i} \cdots u_{k+1}}\,,
\end{eqnarray*}
where $\hat{u}_{i}$ means that this particular argument is omitted. 
For example, for $f\in \cc_{1} $ and $g\in \cc_{2} $ we have 
\begin{eqnarray}\label{e.delta}
\delta f_{st} = f_{t}-f_{s}
\quad
\text{ and }
\quad
\delta g_{sut} = g_{st}-g_{su}-g_{ut}.
\end{eqnarray} 
Let us now  introduce a general notion of  controlled rough process which will be used throughout the paper.

   \begin{Def}\label{def.control}
   Let $x$ and $y,y',y'',\dots, y^{(\ell-1)}$ be real-valued continuous processes on $[ 0,1] $ and assume that the initial values of $y^{(i)}$ equal to zero, namely,   $y^{(i)}_{0}=0$, $i=0,\dots, \ell-1$. Denote the 2-increments:  $x^{i}_{st} = (\delta x_{st})^{i}/i!$, $(s,t)\in\cs_{2}([0,1])$, $i=0,1,\dots,\ell-1$. For convenience, we   also write $y^{(0)} = y$, $y^{(1)} = y'$, $y^{(2)}= y''$,\dots, and $\bfy=(y^{(0)},\dots, y^{(\ell-1)})$. 
    We define the remainder processes 
   \begin{eqnarray}\label{e.r}
 r^{(\ell-1)}_{st} &=& \delta y^{(\ell-1)}_{st} 
    \notag
     \\
 r^{(k)}_{st} &=& \delta y^{(k)}_{st} - y^{(k+1)}_{s} x^{1}_{st}-\cdots  -y^{(\ell-1)}_{s} x^{\ell-k-1}_{st}, \qquad    k=0,1,\dots, \ell-2,
\end{eqnarray}
 for $(s,t) \in \cs_{2}([0,1]) $. 
We call $\bfy$ \emph{a  rough path controlled by $(x, \ell , \al )$ almost surely} for some constant $\al\in(0,1)$  
 if for any $\ep>0$
   there is    
a finite random variable $G_{\bfy}\equiv G_{\bfy,\ep}$ (that is, $G_{\bfy}<\infty$ almost surely)    such that 
 $|r^{(k)}_{st}|\leq G_{\bfy} (t-s)^{(\ell-k)\al-\ep}  $  for all $(s,t)\in\cs_{2}([0,1])$ and $k=0,1,\dots, \ell-1$. 
 We call $\bfy $      \emph{a     rough path  controlled by $(x,\ell, \al )$ in $L_{p}$} for some  $p>0$ if there exist constants $K>0$, $\al\in(0,1)$ such that $|r^{(k)}_{st}|_{L_{p}}\leq K (t-s)^{(\ell-k)\al}$ for all $(s,t)\in\cs_{2}([0,1])$  and $k=0,\dots, \ell-1$. 
\end{Def}

  \begin{remark}
  In some of our computations below we will rephrase \eqref{e.r} for $k=0$ as the following identity for $(s,t)\in \cs_{2}([0,1])$:
  \begin{eqnarray}\label{e.y}
y_{t} &=& \sum_{i=0}^{\ell-1} y^{(i)}_{s} x^{i}_{st} +r^{(0)}_{st},
\end{eqnarray}
where we take $x^{0}\equiv 1$ by convention.  
  \end{remark}
  
In the following we give some    examples of controlled rough paths defined in   Definition \ref{def.control}. 
  
  \begin{example}\label{example.control}
 Take $\al\in(0,1)$. Let  $x_{t}$, $t\in[0,1]$ be a real-valued continuous process whose sample paths are $(\al-\ep)$-H\"older continuous almost surely      for any $\ep>0$. 
For a   continuous function $V$ defined  on $\RR$, we define    the differential operator  $\cl_{V} $ such that  for any differentiable function $f$ we have $\cl_{V} f  = Vf' $. Denote   $\cl_{V}^{i}=\cl_{V}\circ\cdots\circ\cl_{V}$ the $i$th iteration of $\cl_{V}$.

\noindent(i)  Let $V$ be a sufficiently smooth function      on $\RR$. Set $z^{(i)}_{t} = \cl_{V}^{i}V(x_{t})$ for $i=0,\dots, \ell-1$. Then $(z, z', \dots, z^{(\ell-1)})$ is a rough path controlled by $(x,\ell, \al)$ almost surely.

\noindent(ii) 
 Let $ y$ be the solution of the differential equation: $d y_{t} =b(y_{t})dt+  V( y_{t})dx_{t}$ in the sense of \cite[Theorem 12.10]{FV10}, and assume that   the coefficient functions  $b$ and $V$ are  sufficiently smooth.  
Note that when $x$ is a Brownian motion  the differential equation coincides the classical Stratonovich-type SDE. 
   Let $y_{t}'=V(y_{t})$ and $y^{(i)}_{t} = \cl_{V}^{i-1} V(y_{t})$, $i=2,\dots, \ell-1$.  Then $(y, y',\dots, y^{(\ell-1)})$ is a rough path  controlled by $(x,\ell,\al)$ almost surely.

\noindent (iii) Let $\ell=[1/\al]$ and   let $(z, z', \dots, z^{(\ell-1)})$ be a rough path controlled by $(x,\ell, \al)$ almost surely.   Let $y$ be    the rough integral $y_{t}:=\int_{0}^{t}z_{s}dx_{s}$, $t\in[0,1]$  in the   sense of \cite{G04}. An explicit example of rough integral is  $y_{t}=\int_{0}^{t}V(x_{s})dx_{s}$.  Denoting $y'=z$, \dots, $y^{(\ell)}=z^{(\ell-1)}$, then $(y, y',\dots, y^{(\ell)})$ is a rough path controlled by $(x,\ell+1,\al)$ almost surely.  
\end{example}

By  Definition \ref{def.control} it is easy to show  that  the partial sequence   of $\bfy $ and the functions of $y$ are both   controlled rough paths. We state this fact and omit the proof for sake of conciseness:  
  \begin{lemma}\label{lem.control}
Let $\bfy$ be a   rough path controlled by $(x,\ell , \al)$ almost surely (resp., in $L_{p}$ for some $p>0$). Then 

\noindent (i) For any $i=0,\dots, \ell-1$,  $(y^{(i)},   \dots, y^{(\ell-1)})$ is a  rough   path controlled by $(x,\ell-i, \al )$  almost surely (resp., in $L_{p}$).  

\noindent (ii)  Let $f:\RR\to\RR$ be a continuous function which  has   derivatives up to order $(L-1)$ and the $(L-1)$th derivative $f^{(L-1)}$ is Lipschitz.  Let $z^{(0)}_{s} = f(y_{s})$ and 
\begin{eqnarray*}
z_{s}^{(r)}=\sum_{i=1}^{r}  \frac{f^{(i)}(y_{s})}{i!} \sum_{\substack{ 1\leq  j_{1},\dots, j_{i} \leq (L\wedge \ell) -1\\ j_{1}+\cdots+j_{i}=r } }
 \frac{r!}{j_{1}!\cdots j_{i}!} y^{(j_{1})}_{s}\cdots y^{(j_{i})}_{s} 
\end{eqnarray*}
 \text{for } $s\in[0,1]$ and $r=0,\dots, (L\wedge \ell) -1$. For example, we have   $z^{(1)} = f'(y_{s}) y'_{s}$ and $z^{(2)} =  f''(y_{s})(y'_{s})^{2}+f'(y_{s})y''_{s}$. 
 Then $(z^{(0)}, \dots, z^{((L\wedge \ell)-1)})$ is a   rough path controlled by $(x, L\wedge \ell , H)$ almost surely   (resp., in $L_{p}$).
\end{lemma}
  
Let us also    recall an algebraic  result from \cite[Lemma 2.5]{LT20}. 
  \begin{lemma}\label{lem.dr}
Let $x$, $\bfy$ and $r^{(i)}$, $i=0,\dots, \ell-1$ be continuous processes satisfying \eqref{e.r}. Then we have the following relation: 
\begin{eqnarray*}
\delta r^{(0)}_{sut}
  &=&
\sum_{i=1}^{\ell-1} r^{(i)}_{su} x^{i}_{ut} \,. 
\end{eqnarray*}
 
\end{lemma}
  
  
\subsection{Discrete rough integrals}
We introduce some   ``discrete'' integrals defined as Riemann type sums. 
 Namely, let $f$ and $g$ be functions     on   $\cs_{2}([0,1])$.
Let $\cd_{n}=\{ 0=t_{0}<\cdots<t_{n}=1 \}$  be  a generic partition of $[0,1]$. 
  We define  the    discrete  integral of $f$ with respect to $g$ as:
  \begin{eqnarray}\label{e.jfg}
\cj_{s}^{t} (f, g) 
&:=&
\sum_{s\leq t_{k}  <t }   
  f_{s t_{k}} \otimes g_{t_{k}t_{k+1}} , \quad\quad  (s,t) \in \cs_{2} ([0,1]) ,
\end{eqnarray}
where  we  use the convention that $\cj_{s}^{t}(f,g)=0$ whenever $\{t_{k}: s\leq t_{k}<t\} = \emptyset$. 
We highlight   that $f$ in \eqref{e.jfg} is a function of two variables. 
 Similarly, if  $f  $ is a path on $[0,1]$, then we define the  discrete  integral of $f$ with respect to $g$ as:
\begin{eqnarray}\label{e.jfg2}
\cj_{s}^{t} (f, g) 
&:=&
\sum_{s\leq t_{k}  <t } 
   f_{ t_{k}} \otimes g_{t_{k}t_{k+1}}  , \quad\quad  (s,t) \in \cs_{2}([0,1]).
\end{eqnarray}
 

\subsection{Chaos expansion and fractional Brownian motions}
  
  Let $d {\ga} (x) = (2\pi)^{-1/2} e^{-x^{2}/2} dx$ be the standard Gaussian measure on the real line, and let $f \in L_{2}(\ga)$ be such that $\int_{\RR} f(x) d\ga (x) =0$. It is well-known that the function $f$ can be expanded into a series of Hermite polynomials as follows:
  \begin{eqnarray*}
f(x) &=& \sum_{q=d}^{\infty} a_{q} H_{q}(x),
\end{eqnarray*}
where  $d\geq 1$ is some integer and $H_{q} (x) = (-1)^{q} e^{\frac{x^{2}}{2}} \frac{d^{q}}{dx^{q}}  e^{-\frac{x^{2}}{2}}    $ is the Hermite polynomial of order $q$.
Recall that we have the iteration formula: $H_{q+1}(t) = x H_{q}(x)-H_{q}'(x)$.
 If $a_{d}\neq 0$, then $d$ is called the \emph{Hermite rank} of the function~$f$. Since $f \in L_{2}(\ga)$, we   have $\|f\|_{L_{2}(\ga)}^{2}=\sum_{q=d}^{\infty} |a_{q}|^{2} q! <\infty $. 
   
 Let $x$ be a standard fractional Brownian motion (fBm for short) with Hurst parameter $H\in(0,1)$, that is $x$ is a continuous Gaussian process such that 
 \begin{eqnarray*}
\mE[x_{s}x_{t}] = \frac12 (|s|^{2H}+|t|^{2H}-|s-t|^{2H}).
\end{eqnarray*}
  The fBm  $x$ is almost surely $\ga$-H\"older continuous for all $\ga<H$.  Define the covariance function $\rho$ by 
 \begin{eqnarray}\label{e.rho}
\rho (k) = \mE(\delta x_{01}\delta x_{k,k+1}).
\end{eqnarray}
Then, 
 whenever $H <\frac12$, we have $\sum_{k\in \ZZ}\rho (k) = 0$. 
 Let $\ch$ be the completion of the space of indicator functions with respect to the inner product $\langle \mathbf{1}_{[u,v]},  \mathbf{1}_{[s,t]}  \rangle_{\ch} =\mE(\delta x_{uv} \delta x_{st})$.

The following result shows that given two   sequences   of   stable convergence    and  convergence in probability, respectively, their joint sequence is also of stable convergence. Recall that $X_{n}$ is called  convergent to $X$ stably if $(X_{n}, Z)\to(X,Z)$ in distribution as $n\to\infty$ for any $Z\in\cf$.   The reader is referred to \cite{AE, JS} for an introduction to  stable convergence. 
 \begin{lemma}\label{lem.stable}
Let   $Y_{n}^{(1)} $, $Y_{n}^{(2)} $, $n\in\NN$ be two sequences of random variables and   denote the   $\si$-field: $\cf^{Y} = \si\{Y_{n}^{(i)},i=1,2, n\in\NN\}$.  Let  $Y^{(1)}$ be a random variable  
such that the stable convergence  $(Y_{n}^{(1)}, Z)\to (Y^{(1)}, Z)$  as $n\to\infty$ holds for any $Z\in\cf^{Y}$. Suppose     that we have the convergence in probability $Y_{n}^{(2)} \to Y^{(2)} $ as $n\to\infty$ for some random variable  $Y^{(2)}$. Then   the stable convergence $(Y^{(1)}_{n}, Y^{(2)}_{n}, Z)\to (Y^{(1)} , Y^{(2)}, Z )$ as $n\to\infty$  holds for any $Z\in\cf^{Y}$. In particular, we have  the stable convergence  $(Y^{(1)}_{n}+ Y^{(2)}_{n}, Z) \to (Y^{(1)} + Y^{(2)}, Z)$ as $n\to\infty$  for any $Z\in\cf^{Y}$. 
\end{lemma}
\begin{proof}
Since $Y^{(2)}_{n}-Y^{(2)}\to 0$ in probability it follows that the two sequences $(Y^{(1)}_{n}, Y^{(2)}+(Y^{(2)}_{n}-Y^{(2)}), Z)$ and $(Y^{(1)}_{n}, Y^{(2)} , Z)$ have the same limit. On the other hand, by the stable convergence of $Y^{(1)}_{n}$ and the fact that $Y^{(2)}\in\cf^{Y}$ we have the convergence $(Y^{(1)}_{n}, Y^{(2)} , Z) \to (Y^{(1)} , Y^{(2)} , Z)$. We conclude that  the   convergence   
$(Y^{(1)}_{n} , Y^{(2)}_{n} , Z)\to (Y^{(1)} , Y^{(2)} , Z)$ 
as $n\to\infty$ holds. This completes the proof. 
\end{proof}

\section{Upper-bound estimate and limit theorem for weighted random sums}\label{section.bd}

In this section we    derive some     useful     estimates and limit theorem results for   weighted random sums related to fBm.  

%

\subsection{Upper-bound estimate of weighted random sums} 
We   prove  a general       upper-bound estimate result  
for weighted random sums. In the second part of the subsection      we     apply this   estimate result  to  weighted sums involving  fBms. Recall that for a continuous process $x_{t}$, $t\in[0,1]$ and an integer  $i\in\NN$ we denote the 2-increment:   $x^{i}_{st} = (\delta x_{st})^{i}/i!$, $(s,t)\in\cs_{2}([0,1])$.
 

 \begin{prop} \label{prop.yhbd}
Let $x$ be a continuous process on $[0,1]$. 
Let $\bfy= (y^{(0)},\dots, y^{(\ell-1)})$ be a   rough path on $[0,1]$  controlled by $(x,\ell, \al)$ in $L_{2}$ for some $\al>0$ and $\ell\in \NN$, and let   $(r^{(i)}, i=0,\dots, \ell-1)$ be the remainder processes of $\bfy$     defined in Definition \ref{def.control}. Let  $h$ be a  $1$-increment defined on $\cs_{2}(\ll 0, 1 \rr)$.   Let    $\be_{i}\in [0,1]$, $i=0,1,\dots,\ell-1$ be some constants such that    
\begin{eqnarray}\label{e.bei}
\be  := \min_{i=0,\dots, \ell-1} \{ (\ell-i)\al +\be_{i} \}>1 .
\end{eqnarray}
Suppose that there exists a constant $K>0$  such that 
\begin{eqnarray}\label{e.xih}
| \cj_{s}^{t} (x^{i}, h) |_{L_{2}} \leq K (t-s)^{\be_{i}}   
\end{eqnarray}
for  any  $(s,t)\in\cs_{2}([0,1])$ satisfying $t-s\geq 1/n$. 
Then  we can find a constant $K>0$ independent of $n$ such that the following estimates hold:
\begin{eqnarray}\label{e.yhbd1}
| \cj_{s}^{t} ( {r}^{(0)},  h) |_{L_{1}}  \leq   K   (t-s)^{ \be}
\qquad
\text{and}
\qquad
| \cj_{s}^{t} ( y,  h) |_{L_{1}}  \leq   K   (t-s)^{ \be_{0}}
\end{eqnarray}
for    $(s,t)\in\cs_{2}([0,1])$ such that $t-s\geq 1/n$. 
\end{prop}
\begin{proof} 
Denote  $R_{st}:=\cj_{s}^{t} (r^{(0)}, h)$ for $(s,t) \in\cs_{2}([0,1])$.  Recall that the operator $\delta  $ for 2-increment is defined in \eqref{e.delta}. So, for $(s,u,t)\in \cs_{3}([0,1])$, we have 
\begin{eqnarray}\label{e.dR1}
\delta R_{sut} &=&   \cj_{s}^{t} ( {r}^{(0)},  h) -\cj_{s}^{u} (r^{(0)},  h)-\cj_{u}^{t} (r^{(0)},  h)  
\nonumber 
\\
&=& \sum_{u\leq t_{k}<t} ( r^{(0)}_{st_{k}} - r^{(0)}_{ut_{k}} ) h_{t_{k}t_{k+1}}
 .                                 
\end{eqnarray}
Note that  by definition of $\delta r^{(0)}$ we have the relation   $r^{(0)}_{st_{k}} -r^{(0)}_{ut_{k}} = \delta r^{(0)}_{sut_{k}} +r^{(0)}_{su} $. Substituting this   into \eqref{e.dR1} and then    invoking Lemma \ref{lem.dr}  
 we   obtain 
\begin{eqnarray}\label{e.dR}
\delta R_{sut}&=&  r^{(0)}_{su} \cj_{u}^{t}(  1 , h )+\sum_{i=1}^{\ell-1} r^{(i)}_{su} \cj_{u}^{t}(   x^{i}  , h ).
\end{eqnarray}
 We can now bound $\delta R$ as follows:  Taking the $L_{1}$-norm on both sides of \eqref{e.dR} gives 
 \begin{eqnarray}\label{e.drl1}
|\delta R_{sut}|_{L_{1}} \leq \sum_{i=0}^{\ell-1} |r^{(i)}_{su}|_{L_{2}} \cdot | \cj_{u}^{t}(   x^{i}  ,  h )|_{L_{2}} .  
\end{eqnarray}
Applying  condition  \eqref{e.xih} to $| \cj_{u}^{t}(   x^{i}  ,  h )|_{L_{2}}$ in \eqref{e.drl1} and invoking the relation  $|r^{(i)}_{st}|_{L_{p}}\leq K (t-s)^{(\ell-i)\al}$ given in  Definition \ref{def.control} we get  
\begin{eqnarray}\label{e.drbd}
|\delta R_{sut}|_{L_{1}}  \lesssim  \sum_{i=0}^{\ell-1} (u-s)^{(\ell -i)\al} (t-u)^{\be_{i}}  \lesssim  (t-s)^{\be} 
\end{eqnarray}
for $(s,u,t)\in \cs_{3}([0,1])$ such that $t-u\geq 1/n$, 
where   $\be$ is defined in \eqref{e.bei}. 

Take $(s,t)\in\cs_{2}([0,1])$ such that $t-s\geq 1/n$. 
Consider the partition $\ll s,t \rr  $ of the interval $[s,t]$: $s<t_{k}<\cdots<t_{k'}<t$, where $k$ and $k'$ are such that $t_{k-1}\leq s<t_{k}$ and $t_{k'}<t\leq t_{k'+1} $.  
In the following we show    that \eqref{e.drbd} holds for all $(u_{1}, u_{2}, u_{3})\in \cs_{3}(\ll s,t \rr)$.  In view of \eqref{e.drbd} it remains to show that the estimate \eqref{e.drbd} holds for $|\delta R_{ut_{k'}t}|_{L_{1}}$, $u\in \ll s,t\rr:u\leq t_{k'}$.  Indeed, 
by definition \eqref{e.jfg} we have $\cj_{t_{k'}}^{t}(x^{i}, h)=0$ and 
\begin{eqnarray*}
|\cj_{t_{k'}}^{t}(1, h)|_{L_{2}}=|h_{t_{k'}t_{k'+1}}|_{L_{2}}\leq (1/n)^{\be_{0}}\leq (t-s)^{\be_{0}}. 
\end{eqnarray*}
 Applying these estimates to the right-hand side of \eqref{e.drl1}  we obtain   
 the estimate \eqref{e.drbd}   for $|\delta R_{ut_{k'}t}|_{L_{1}}$. 

By   \eqref{e.jfg}    it is clear that  for any two consecutive partition points $u,v$ in $\ll s,t \rr$ and $u<v$   we have $R_{uv} = 0$. 
Applying the discrete sewing  lemma    \cite[Lemma 2.5]{LT20} to $R$ on the partition $\ll s,t \rr$ and then invoking the estimate \eqref{e.drbd} of $\delta R$  on $\cs_{3}(\ll s,t\rr)$   we obtain  
\begin{eqnarray*}
|  R_{st}|_{L_{1}} \lesssim (t-s)^{\be} .
\end{eqnarray*}
 This proves    the first   estimate in~\eqref{e.yhbd1}.  

Note  that by substituting the expression  \eqref{e.y} into $\cj_{s}^{t} (y, h)$ we get the relation 
 \begin{eqnarray}\label{e.yh}
\cj_{s}^{t} (y, h) =   \sum_{i=0}^{\ell-1} y^{(i)}_{s} \cj_{s}^{t}(x^{i}, h) + \cj_{s}^{t}(r^{(0)}, h) . 
\end{eqnarray}
 Applying \eqref{e.xih} and the first estimate in \eqref{e.yhbd1} to the right-hand side of \eqref{e.yh} we obtain  the  desired estimate of $\cj_{s}^{t} (y, h)$ in~\eqref{e.yhbd1}. 
\end{proof}

In the next result we   apply Proposition \ref{prop.yhbd} to weighted sums which involve  fBms. 

\begin{prop}\label{prop.yhbd2}
Let $x$ be a one-dimensional fBm with Hurst parameter $H \leq 1/2$. 
 Suppose that  $ (y ,y' , \dots, y^{(\ell-1)})$, $\ell\in\NN$ is a   process controlled by $(x,\ell, H-\ep)$ in $L_{2}$ for some sufficiently small  $\ep>0$.
Let  $f= \sum_{q=d}^{\infty} a_{q}H_{q} \in L_{2}(\RR,\ga)$ with  Hermite rank   $d>0$ and $f$   belongs to  the Soblev space $W^{2(\ell-1), 2}( \RR, \ga)$, where  $\ga $ denotes the standard Gaussian measure on the real line;   see e.g. Page 28 in \cite{N06}.  
   We define a family of increments $\{h^{n}; n\geq 1 \}$ by:
 \begin{eqnarray}\label{e.hst}
h^{n}_{st} :=  \sum_{s\leq t_{k}<t} f(n^{H} \delta x_{t_{k}t_{k+1}}), \quad\quad (s,t)\in \cs_{2}(\ll 0,1 \rr). 
\end{eqnarray}

\noindent (i) Suppose that $d> \frac{1}{2H}$ and that  
  $\ell  $ is  the least  integer such that $  \ell H+\frac12>1$, that is $\ell=[\frac{1}{2H}]+1$.  
Then there is a constant $K$ independent of $n$  such that  
\begin{eqnarray}\label{e.yhbdi}
|  \cj_{s}^{t} ( y , h^{n}  )  |_{L_{1}}  &\leq& K  n^{1/2} (t-s)^{ 1/2 }    
\end{eqnarray}
 for all $(s,t)\in \cs_{2}([ 0,1 ]) $ satisfying  $ t-s\geq1/n$. 

\noindent (ii) Suppose that $d\leq  \frac{1}{2H}$ and that  
  $\ell  =d+1$.   
Then  there is a constant $K$ independent of $n$ such that 
\begin{eqnarray}\label{e.yhbdii}
|  \cj_{s}^{t} ( y , h^{n}  )  |_{L_{1}}  &\leq& K  n^{1-dH} (t-s)^{1-dH}   
\end{eqnarray}
 for all $(s,t)\in \cs_{2}([ 0,1 ]) $  satisfying  $ t-s \geq1/n$. 
\end{prop}
\begin{proof}
   We assume that $d>\frac{1}{2H}$. In the following we prove (i) by applying Proposition \ref{prop.yhbd}. We first recall the estimate  in \cite[equation (4.24)]{LT20}: 
 \begin{eqnarray}\label{e.xihn}
| \cj_{s}^{t} (x^{i}, h^{n}) |_{L_{2}} \leq Kn^{1/2} (t-s)^{iH+1/2} , 
\qquad
\text{ $(s,t)\in \cs_{2}([ 0,1 ]): t-s\geq 1/n $, }
\end{eqnarray} 
for all $i=0,\dots,[\frac{1}{2H}]$.  The estimate \eqref{e.xihn} implies that  
 relation  \eqref{e.xih} holds for  $h := h^{n}/\sqrt{n}$ and $\be_{i}:=iH+1/2$.  
Take $\al=H-\ep$ and recall that $\ell$ is the least integer such that $\ell H +1/2>1$, or  
  $\ell H>1/2$.  It is thus readily checked that  
  condition \eqref{e.bei} is satisfied. We conclude that \eqref{e.yhbdi} holds.   

 We turn to the case when  $d\leq \frac{1}{2H}$. As before, our estimate will be an application of  Proposition \ref{prop.yhbd}.  We first derive an estimate of $| \cj_{s}^{t} (x^{i}, h^{n}) |_{L_{2}} $. For this purpose we   consider the following  decomposition
 \begin{eqnarray}\label{e.hd}
h^{n}_{st} = h^{n,(1)}_{st}+h^{n,(2)}_{st},
\end{eqnarray}
where
\begin{eqnarray*}
  h^{n,(1)}_{st} =  \sum_{s\leq t_{k}<t} f_{1}(n^{H} \delta x_{t_{k}t_{k+1}}) , 
\qquad
 h^{n,(2)}_{st} = \sum_{s\leq t_{k}<t} f_{2}(n^{H} \delta x_{t_{k}t_{k+1}})  
  \end{eqnarray*}
and 
\begin{eqnarray*}
f_{1}(x) =  \sum_{q=[\frac{1}{2H}]+1}^{\infty} a_{q}H_{q} (x) , 
\qquad
\qquad
f_{2}(x) =  \sum_{q=d}^{[\frac{1}{2H}]}  a_{q}H_{q}(x)  . 
\end{eqnarray*}
Note that the Hermite rank of $f_{1}$ is greater than $\frac{1}{2H}$. So  we can  apply   \eqref{e.xihn}     to get   the estimate 
\begin{eqnarray}\label{e.xh1}
| \cj_{s}^{t} (x^{i}, h^{n,(1)}) |_{L_{2}} \leq Kn^{1/2} (t-s)^{iH+1/2}   
\end{eqnarray}
for  $(s,t)\in \cs_{2}([ 0,1 ]): t-s\geq 1/n $ and $i=0,\dots, d$. 
By assumption  we have $1/2-dH>0$. It follows that  
\begin{eqnarray}\label{e.nst}
1\leq  n^{1/2-dH}(t-s)^{1/2-dH},  
\end{eqnarray}
and therefore we can enlarge the bound in \eqref{e.xh1} to be:    
\begin{eqnarray}\label{e.xih1}
| \cj_{s}^{t} (x^{i}, h^{n,(1)}) |_{L_{2}}  \leq Kn^{1-dH} (t-s)^{ 1+iH-dH }  . 
\end{eqnarray}
 
Let us turn to the estimate of $| \cj_{s}^{t} (x^{i}, h^{n,(2)}) |_{L_{2}} $. We first have the bound
\begin{eqnarray}\label{e.xih2}
| \cj_{s}^{t} (x^{i}, h^{n,2}) |_{L_{2}} \leq \sum_{q=d}^{[\frac{1}{2H}]}  |a_{q} | \cdot | \cj_{s}^{t} (x^{i}, h^{n,q}) |_{L_{2}} 
, \qquad 
(s,t)\in \cs_{2}([ 0,1 ])  . 
\end{eqnarray}
 Recall the estimate in  \cite[Lemma 4.11 (ii)]{LT20}:
\begin{eqnarray}\label{e.xihq}
|\cj_{s}^{t}(x^{i}, h^{n,q})|_{L_{2}} \lesssim 
\begin{cases}
n^{1-qH } (t-s)^{1+iH-qH} & \text{when } q\leq i
\\
n^{1/2}(t-s)^{iH+1/2}&\text{when }  q>i
\end{cases}
\end{eqnarray}
for $(s,t)\in\cs_{2}([0,1]):t-s\geq1/n$, $i=0,\dots, d$ and  $q<\frac{1}{2H}$. 
Substituting  \eqref{e.xihq}  
into the right-hand side of \eqref{e.xih2} and then 
applying    the relation $1\leq n(t-s) $   we    obtain that  
\begin{eqnarray}\label{e.xh2}
| \cj_{s}^{t} (x^{i}, h^{n,2}) |_{L_{2}}   \leq Kn^{1-dH} (t-s)^{ 1+iH-dH } , \qquad 
(s,t)\in \cs_{2}([ 0,1 ])   
\end{eqnarray}
for $i=0,\dots, d $.

Combining  the two  estimates \eqref{e.xih1} and \eqref{e.xh2} and taking into account  the decomposition \eqref{e.hd},  we obtain that \eqref{e.xih} holds for $\be_{i} := 1+iH-dH $ and  $h := h^{n}/n^{1-dH}$. It is readily checked that condition    \eqref{e.bei} is satisfied  for $\ell =d+1$. Applying Proposition \ref{prop.yhbd}
we thus conclude the desired estimate \eqref{e.yhbdii}.  
\end{proof}
\subsection{Convergence of Riemann sum}

Let $y$ be a  continuous process controlled by the fBm~$x$. This subsection is devoted to the convergence of   Riemann sum for  the regular integral   $\int_{0}^{t}y_{u}du$. For convenience we will consider   the uniform partition of $[0,1]$: $t_{i}=i/n$, $ i=0,1,\dots, n.$ 

We   start by proving  the following    weighted     limit theorem result:  
\begin{lemma}\label{lem.yxdt}
Let $x$ be a one-dimensional fBm with Hurst parameter $H < 1/2$.  Let $(y,y')$ be a   rough path controlled by $(x, 2, H)$ almost surely. 
Define the increment 
\begin{eqnarray}\label{e.hn}
h^{n}_{st} =  \sum_{s\leq t_{k}<t}   \int_{t_{k}}^{t_{k+1}} x^{1}_{t_{k}u} du 
\qquad\text{for } (s,t)\in\cs_{2} (\ll0,1\rr). 
\end{eqnarray}
 Then for each   $(s,t)\in\cs_{2} ([0,1])$
   we have the    convergence in probability:  
\begin{eqnarray}\label{e.yhc}
n^{2H}\cj_{s}^{t}(y, h^{n}) \to  -\frac{1}{4H+2} \int_{s}^{t}y'_{u}du
\qquad \text{as $n\to\infty$} .
\end{eqnarray} 
\end{lemma}
\begin{proof}
The proof is divided into several steps. By localization (cf. \cite[Lemma 3.4.5]{JP}) we can and will assume that $(y,y')$ is controlled by $(x,2,H-\ep)$ in $L_{2}$ for any $\ep>0$. 

\noindent\emph{Step 1: Estimate  of $h^{n}$.} 
By the self-similarity of the fBm we have   $\mE [x^{1}_{t_{k}u}x^{1}_{t_{k'}u'}] = n^{-2H}
\mE [x^{1}_{{k}, nu}x^{1}_{{k'},nu'}] 
  $. 
Applying  this relation   and then the change of variable $nu'\to u'$ and $nu\to u$  we get 
\begin{eqnarray*}
\mE[|h^{n}_{st}|^{2} ] 
&=& n^{-2H-2} \sum_{ns\leq k,k'<nt}   \int_{k}^{k+1} \int_{ {k'}}^{ {k'+1}} 
\mE [x^{1}_{ {k}u}x^{1}_{{k'}u'}] 
du' du, 
\qquad  (s,t)\in\cs_{2} (\ll0,1\rr).
\end{eqnarray*}
 Applying the estimate  $|\mE [x^{1}_{ {k}u}x^{1}_{{k'}u'}] |\lesssim |k-k'|^{2H-2}$ for $k\neq k'$ we obtain 
 \begin{eqnarray}\label{e.hnbd}
\mE[|h^{n}_{st}|^{2} ]  \lesssim n^{-2H-2} \sum_{\substack{ns\leq k,k'<nt\\k\neq k'}}
  |k-k'|^{2H-2} 
\lesssim n^{-2H-1} (t-s) , 
\qquad  (s,t)\in\cs_{2} (\ll0,1\rr). 
\end{eqnarray}
 
\noindent\emph{Step 2: A decomposition    of $|\cj_{s}^{t}(x^{1}, h^{n}) |_{L_{2}}^{2}$.} 
Let  $   (s,t)\in\cs_{2} ([0,1])$ such that $t-s>1/n$. 
By definition \eqref{e.jfg2} we can   express $|\cj_{s}^{t}(x^{1}, h^{n}) |_{L_{2}}^{2}$ as 
\begin{eqnarray}\label{e.jxih}
|\cj_{s}^{t}(x^{1}, h^{n}) |_{L_{2}}^{2}
&=&  \sum_{s\leq t_{k},t_{k'}<t}   \mE  
\int_{t_{k}}^{t_{k+1}}\int_{t_{k'}}^{t_{k'+1}} x^{1}_{st_{k}}x^{1}_{st_{k'}} x^{1}_{t_{k}u} x^{1}_{t_{k'}u'}du' du . 
\end{eqnarray}
Applying the integration by part  to the integrand in \eqref{e.jxih} we obtain   
\begin{eqnarray}\label{e.xha}
\mE \lp
x^{1}_{st_{k}}x^{1}_{st_{k'}} x^{1}_{t_{k}u} x^{1}_{t_{k'}u'}  
\rp = A_{1}+A_{2},
\end{eqnarray}
where
\begin{eqnarray}
A_{1} &=&   \mE \lp
x^{1}_{st_{k}}x^{1}_{st_{k'}} \rp
 \langle \mathbf{1}_{[t_{k},u]}, \mathbf{1}_{[t_{k'},u']} \rangle_{\ch}
 \label{e.a1i}
 \\
 A_{2} &=&    \langle D^{2}(x^{1}_{st_{k}}x^{1}_{st_{k'}}), \mathbf{1}_{[t_{k},u]}\otimes  \mathbf{1}_{[t_{k'},u']}  \rangle_{\ch^{\otimes 2}}. 
 \label{e.a2}
\end{eqnarray}

\noindent\emph{Step 3: Estimate of $A_{1}$.}
It is clear that $\Big| \mE \lp
x^{1}_{st_{k}}x^{1}_{st_{k'}} \rp\Big|\leq (t-s)^{2H}$. On the other hand, similar to the estimate of $\mE[\delta x_{t_{k}u}\delta x_{t_{k'}u'}]$ in Step 1 
we have $|\langle \mathbf{1}_{[t_{k},u]}, \mathbf{1}_{[t_{k'},u']} \rangle_{\ch}| \lesssim n^{-2H} |k-k'|^{2H-2} $. Substituting these two estimates into \eqref{e.a1i} we obtain 
\begin{eqnarray*}
|A_{1}|\lesssim (t-s)^{2H} |k-k'|^{2H-2} n^{-2H} . 
\end{eqnarray*}
The above   estimate for $|A_{1}|$ together with  the relation \begin{eqnarray*}
\sum_{s\leq t_{k}, t_{k'}< t} |k-k'|^{2H-2}\lesssim n(t-s) 
\end{eqnarray*}
shows that 
 \begin{eqnarray}\label{e.a1bd}
  \sum_{s\leq t_{k},t_{k'}<t}      
\int_{t_{k}}^{t_{k+1}}\int_{t_{k'}}^{t_{k'+1}} A_{1} du' du 
  \lesssim (t-s)^{2H+1}  n^{-1-2H} 
. 
\end{eqnarray}

\noindent\emph{Step 4: Estimate of $A_{2}$.}
Recall that $A_{2}$ is defined in \eqref{e.a2}.
We first note that  
\begin{eqnarray}\label{e.a2d}
D^{2}(x^{1}_{st_{k}}x^{1}_{st_{k'}}) =   \mathbf{1}_{[s,t_{k}]}\otimes \mathbf{1}_{[s,t_{k'}]} +\mathbf{1}_{[s,t_{k'}]}\otimes \mathbf{1}_{[s,t_{k}]} . 
\end{eqnarray}
Substituting relation \eqref{e.a2d} into \eqref{e.a2}   we obtain the decomposition 
\begin{eqnarray}\label{e.a2de}
 \sum_{s\leq t_{k},t_{k'}<t}      
\int_{t_{k}}^{t_{k+1}}\int_{t_{k'}}^{t_{k'+1}} A_{2} du' du 
  = A_{21}+A_{22}\,, 
\end{eqnarray}
where
 \begin{eqnarray}
A_{21}&=&\sum_{s\leq t_{k},t_{k'}<t}      
\int_{t_{k}}^{t_{k+1}}\int_{t_{k'}}^{t_{k'+1}}  
  \langle \mathbf{1}_{[s,t_{k'}]} \otimes  \mathbf{1}_{[s,t_{k}]}  , \mathbf{1}_{[t_{k},u]}\otimes  \mathbf{1}_{[t_{k'},u']}  \rangle_{\ch^{\otimes 2}}
 du' du 
 \nonumber
 \\
A_{22}&=& \sum_{s\leq t_{k},t_{k'}<t}      
\int_{t_{k}}^{t_{k+1}}\int_{t_{k'}}^{t_{k'+1}}  
  \langle \mathbf{1}_{[s,t_{k}]} \otimes  \mathbf{1}_{[s,t_{k'}]}  , \mathbf{1}_{[t_{k},u]}\otimes  \mathbf{1}_{[t_{k'},u']}  \rangle_{\ch^{\otimes 2}}
 du' du . 
 \label{e.a22}
\end{eqnarray}
In the following we  bound $A_{21}$ and $A_{22}$, which together will give us the estimate of $A_{2}$.  We   first have 
\begin{eqnarray*}
|A_{21}| \leq 2\sum_{s\leq t_{k}\leq t_{k'}<t}   
\int_{t_{k}}^{t_{k+1}}\int_{t_{k'}}^{t_{k'+1}}  
 | \langle \mathbf{1}_{[s,t_{k'}]}  , \mathbf{1}_{[t_{k},u]}   \rangle_{\ch }|\cdot 
  | \langle   \mathbf{1}_{[s,t_{k}]}  ,  \mathbf{1}_{[t_{k'},u']}  \rangle_{\ch}|
 du' du .
\end{eqnarray*}
Invoking   the  elementary  estimates  
\begin{eqnarray*}
 | \langle \mathbf{1}_{[s,t_{k'}]}  , \mathbf{1}_{[t_{k},u]}   \rangle_{\ch }|
\leq n^{-2H} 
  \qquad\text{and}  \qquad
  | \langle   \mathbf{1}_{[s,t_{k}]}  ,  \mathbf{1}_{[t_{k'},u']}  \rangle_{\ch}|\lesssim n^{-2H}|k-k'|^{2H-1}
\end{eqnarray*}
for $t_{k}\leq t_{k'}$ we obtain 
  \begin{eqnarray}
|A_{21}| \lesssim \sum_{s\leq t_{k}\leq t_{k'}<t}   
\int_{t_{k}}^{t_{k+1}}\int_{t_{k'}}^{t_{k'+1}}  
 n^{-2H} \cdot 
  n^{-2H}\cdot   (k'-k)^{2H-1}
 du' du 
 \nonumber
 \\
 \lesssim  (t-s)^{2H+1} n^{-2H-1}. 
 \label{e.a21bd}
\end{eqnarray}
We turn to the estimate of  $A_{22}$.   A change of variables in \eqref{e.a22} gives   
\begin{eqnarray}\label{e.a22i}
A_{22}&=&n^{-4H-2}\sum_{ns\leq {k}, {k'}<nt}      
A_{22,kk'}
 , 
\end{eqnarray}
where
\begin{eqnarray}\label{e.a22k}
A_{22,kk'} = \int_{ {k}}^{ {k+1}}\int_{ {k'}}^{ {k'+1}}  
  \langle \mathbf{1}_{[ns, {k}]} \otimes  \mathbf{1}_{[ns, {k'}]}  , \mathbf{1}_{[ {k},u]}\otimes  \mathbf{1}_{[ {k'},u']}  \rangle_{\ch^{\otimes 2}}
 du' du . 
\end{eqnarray}
It is clear that $|A_{22,kk'}|\sim O(1)$. Therefore, from \eqref{e.a22i} we obtain  the estimate 
\begin{eqnarray}\label{e.a22bd}
|A_{22}|\lesssim (t-s)^{2}n^{-4H}. 
\end{eqnarray}

\noindent\emph{Step 5: Estimate  of   $\cj_{s}^{t}(x^{1},h^{n})$.}
Putting together the estimates  \eqref{e.a1bd}, \eqref{e.a21bd}, and \eqref{e.a22bd} and taking into account the decompositions \eqref{e.jxih}-\eqref{e.xha} and \eqref{e.a2de} we obtain the estimate
\begin{eqnarray}\label{e.xhl2}
|\cj_{s}^{t}(x^{1}, h^{n}) |_{L_{2}} \lesssim (t-s)n^{-2H} 
\end{eqnarray}
for $   (s,t)\in\cs_{2} ([0,1])$ such that $t-s>1/n$. 

\noindent\emph{Step 6: Convergence  of the second moment  of   $\cj_{s}^{t}(x^{1},h^{n})$.} Let $   (s,t)\in\cs_{2} ([0,1])$.  In this  step we   show the  convergence: 
\begin{eqnarray}\label{e.x1hl2}
n^{2H}| \cj_{s}^{t}(x^{1},h^{n})|_{L_{2}} \to \frac{1}{4H+2}(t-s) 
\qquad\text{as $n\to\infty$.}
\end{eqnarray}

Recall  our decomposition of $|\cj_{s}^{t}(x^{1},h^{n})|_{L_{2}}^{2}$ in \eqref{e.jxih}-\eqref{e.xha} and of $A_{2}$ in \eqref{e.a2de}. So   the estimates in \eqref{e.a1bd}, \eqref{e.a21bd}   and \eqref{e.a22bd} together shows that  the convergence of $|\cj_{s}^{t}(x^{1},h^{n})|_{L_{2}}^{2}$ is dominated by that of $A_{22}$. Namely, we have
\begin{eqnarray}\label{e.x1ha11}
\lim_{n\to\infty}n^{4H}| \cj_{s}^{t}(x^{1},h^{n})|_{L_{2}}^{2}   = \lim_{n\to\infty}n^{4H} A_{22}. 
\end{eqnarray}
In the following we focus on the computation of         $\lim_{n\to\infty}n^{4H}A_{22}$. 

Recall the expression of $A_{22}$ in  \eqref{e.a22i}-\eqref{e.a22k}. 
We first note that since $|A_{22,kk'}|\sim O(1)$ we can replace the summation $\sum_{ns\leq {k}, {k'}<nt}     $ in \eqref{e.a22i} by $\sum_{ns+n^{\ep}\leq {k}, {k'}<nt }     $ for   $0<\ep<1$ without changing the limit of $A_{22}$. Next, by stationary increment and  self-similarity of the fBm we have
\begin{eqnarray*}
\langle \mathbf{1}_{[ns, {k}]} \otimes  \mathbf{1}_{[ns, {k'}]}  , \mathbf{1}_{[ {k},u]}\otimes  \mathbf{1}_{[ {k'},u']}  \rangle_{\ch^{\otimes 2}} = 
\langle \mathbf{1}_{[ns-k, 0]} \otimes  \mathbf{1}_{[ns-k',0]}  , \mathbf{1}_{[ 0,u-k]}\otimes  \mathbf{1}_{[ 0,u'-k']}  \rangle_{\ch^{\otimes 2}}
\\
=\langle \mathbf{1}_{[\frac{ns-k}{u-k}, 0]} \otimes  \mathbf{1}_{[\frac{ns-k'}{u'-k'},0]}  , \mathbf{1}_{[ 0,1]}\otimes  \mathbf{1}_{[ 0,1]}  \rangle_{\ch^{\otimes 2}}(u-k)^{2H}(u'-k')^{2H}
\\
=\langle \mathbf{1}_{(-\infty , 0]} \otimes  \mathbf{1}_{(-\infty,0]}  , \mathbf{1}_{[ 0,1]}\otimes  \mathbf{1}_{[ 0,1]}  \rangle_{\ch^{\otimes 2}}(u-k)^{2H}(u'-k')^{2H}+o(1),
\end{eqnarray*}
where the last equation holds for $k$ and $k'$ such that $k-ns \geq  n^{\ep}$ and $ k'-ns\geq  n^{\ep}$. Using the relation $\langle \mathbf{1}_{(-\infty, 0]},  \mathbf{1}_{[  0, 1]}  \rangle_{\ch} = -1/2$ we obtain     
\begin{eqnarray}\label{e.a22ii}
\langle \mathbf{1}_{[ns, {k}]} \otimes  \mathbf{1}_{[ns, {k'}]}  , \mathbf{1}_{[ {k},u]}\otimes  \mathbf{1}_{[ {k'},u']}  \rangle_{\ch^{\otimes 2}} =  \frac14 (u-k)^{2H}(u'-k')^{2H}+o(1).
\end{eqnarray}
Substituting \eqref{e.a22ii} into \eqref{e.a22i} we obtain 
\begin{eqnarray*}
A_{22}
  &=&
  n^{-4H-2}\frac14 \sum_{ns+n^{\ep}\leq {k}, {k'}<nt}      
\int_{ {k}}^{ {k+1}}\int_{ {k'}}^{ {k'+1}}  
   (u-k)^{2H}   (u'-k')^{2H}
 du' du+ n^{-4H}o(1)
 \\
 &=&n^{-4H}(t-s)^{2} \cdot  \frac14  (2H+1)^{-2}+ n^{-4H}o(1). 
\end{eqnarray*}
It follows that 
\begin{eqnarray*}
\lim_{n\to\infty}  n^{4H}A_{22} =  (t-s)^{2} \cdot  \frac14  (2H+1)^{-2} .
\end{eqnarray*}
Recalling relation  \eqref{e.x1ha11}, we thus obtain  the   convergence in \eqref{e.x1hl2}. 

\noindent\emph{Step 7: Convergence    of   $\cj_{s}^{t}(x^{1},h^{n})$.}
In this step, we show the $L_{2}$-convergence of $\cj_{s}^{t}(x^{1},h^{n})$: 
\begin{eqnarray}\label{e.x1hconv}
n^{2H}\cj_{s}^{t}(x^{1},h^{n}) \to -\frac{1}{4H+2}(t-s) .
\end{eqnarray}
In view of the convergence \eqref{e.x1hl2},   it suffices  to  show that:  
\begin{eqnarray}\label{e.xh1m}
n^{2H} \mE  \cj_{s}^{t}(x^{1},h^{n})  \to -\frac{1}{4H+2}(t-s)
\qquad\text{as $n\to\infty$. }
\end{eqnarray}
The convergence \eqref{e.xh1m} can be proved  in  the similar way as  in Step 5. Indeed,  we have:  
\begin{eqnarray*}
\mE [\cj_{s}^{t}(x^{1}, h^{n})]&=&  \mE\sum_{s\leq t_{k}<t}  x^{1}_{st_{k}} \int_{t_{k}}^{t_{k+1}} x^{1}_{t_{k}u} du
\\
&=&  \sum_{s\leq t_{k}<t}   \int_{t_{k}}^{t_{k+1}} \langle \mathbf{1}_{[s,t_{k}]}, \mathbf{1}_{[t_{k}, u]} \rangle_{\ch} du
\\
&=& n^{-2H-1} \sum_{ns\leq {k}<nt}   \int_{{k}}^{{k+1}} \langle \mathbf{1}_{[ns,{k}]}, \mathbf{1}_{[ {k}, u]} \rangle_{\ch} du
\\
&=&  n^{-2H-1} (-1/2)\cdot \sum_{ns\leq {k}<nt}   \int_{{k}}^{{k+1}}  
(u-k)^{2H}  du +n^{-2H}o(1)
\\
&=&n^{-2H} (t-s)  (-1/2)(2H+1)^{-1}+n^{-2H}o(1). 
\end{eqnarray*}
The    convergence   \eqref{e.xh1m} then follows. The two convergences   
\eqref{e.x1hl2} and \eqref{e.xh1m} together implies that  the     convergence  \eqref{e.x1hconv} holds. 

\noindent\emph{Step 8: Convergence  of $\cj_{s}^{t}(y,h^{n})$.} Let $   (s,t)\in\cs_{2} ([0,1])$.  
We start by   taking a   partition of $[s,t]$: $s=s_{0}<s_{1}<\cdots<s_{m}=t$    such that $\max_{j=0,\dots,m-1}|s_{j+1}-s_{j}|\leq 1/m$ for some $m<n$. Then we can write
\begin{eqnarray}\label{e.cjsj}
\cj_{s}^{t}(y, h^{n}) = \sum_{j=0}^{m-1} \cj_{s_{j}}^{s_{j+1}}(y, h^{n}) . 
\end{eqnarray}
Since $(y,y')$ is controlled by $(x,2,H-\ep)$ in $L_{2}$ we have the expansion $y_{t_{k}} = y_{s_{j}}+y_{s_{j}}'x^{1}_{s_{j}t_{k}} +r^{(0)}_{s_{j}t_{k}}$. Substituting this into $\cj_{s_{j}}^{s_{j+1}}(y, h^{n}) $ in \eqref{e.cjsj} we obtain 
\begin{eqnarray}\label{e.yhde}
\cj_{s}^{t}(y, h^{n}) = \sum_{j=0}^{m-1} y_{s_{j}}\cj_{s_{j}}^{s_{j+1}}(1, h^{n})  + \sum_{j=0}^{m-1} y_{s_{j}}'\cj_{s_{j}}^{s_{j+1}}(x^{1}, h^{n}) + \sum_{j=0}^{m-1}  \cj_{s_{j}}^{s_{j+1}}(r^{(0)} , h^{n})  . 
\end{eqnarray}
In the following  we consider the convergence of the three terms on the right-hand side of \eqref{e.yhde}. 

We   note that it follows from   relations \eqref{e.hnbd}  and \eqref{e.xhl2}    that conditions  \eqref{e.bei}-\eqref{e.xih} hold for  $h:=n^{2H}h^{n}$, $\al=H-\ep$, $\be_{0}:=1-H$ and $\be_{1}:=1$. Therefore, applying Proposition \ref{prop.yhbd} we have 
\begin{eqnarray*}
n^{2H}|\cj_{s}^{t}(r^{(0)}, h^{n})|_{L_{1}}\lesssim (t-s)^{1+H-\ep}. 
\end{eqnarray*}
This implies that
\begin{eqnarray}\label{e.rhc}
\lim_{m\to\infty} \limsup_{n\to\infty} n^{2H}\Big|\sum_{j=0}^{m-1}  \cj_{s_{j}}^{s_{j+1}}(r^{(0)}, h^{n})\Big|_{L_{1}} \lesssim  \lim_{m\to\infty} \sum_{j=0}^{m-1}  (s_{j+1}-s_{j})^{1+H-\ep} =0. 
\end{eqnarray}
We turn to the other two  terms in the right-hand side of \eqref{e.yhde}. Applying \eqref{e.hnbd} we have 
\begin{eqnarray}\label{e.hw}
 &&
 n^{2H}\Big|\sum_{j=0}^{m-1} y_{s_{j}}\cj_{s_{j}}^{s_{j+1}}(1, h^{n})\Big|_{L_{1}} =n^{2H}\sum_{j=0}^{m-1} \Big|y_{s_{j}}   h^{n}_{s_{j} s_{j+1}}\Big|_{L_{1}}
\notag
\\
 &&\qquad\lesssim \sum_{j=0}^{m-1} |y_{s_{j}} |_{L_{2}} (s_{j+1}-s_{j})^{1/2}(1/n)^{1/2-H} \to 0 
 \qquad
 \text{as $n\to\infty$. }
\end{eqnarray}
Finally, according to \eqref{e.x1hconv} we have the   convergence in probability: 
\begin{eqnarray}\label{e.x1hw}
\lim_{m\to\infty}
\lim_{n\to\infty}
n^{2H}\sum_{j=0}^{m-1} y_{s_{j}}'\cj_{s_{j}}^{s_{j+1}}(x^{1}, h^{n}) 
&=&
-\frac{1}{4H+2} \lim_{m\to\infty} \sum_{j=0}^{m-1} y_{s_{j}}'  (s_{j+1}-s_{j}) 
\notag
\\
&=& -\frac{1}{4H+2} \int_{0}^{t}y_{u}' du. 
\end{eqnarray}
Putting together the convergences \eqref{e.rhc}-\eqref{e.x1hw} and recalling the relation \eqref{e.yhde} we conclude the convergence \eqref{e.yhc}.  
\end{proof}

With Lemma \ref{lem.yxdt} in hand,  we are ready to consider the convergence  of the Riemann sum for the integral $\int_{0}^{t}y_{u}du$.   

  \begin{prop}\label{prop.riemann}
Let $x$ be a one-dimensional fBm with Hurst parameter $H < 1/2$.  Let $(y,y',y'')$ be a   rough path controlled by $(x, 3, H)$ almost surely.
  Then we have the convergence in probability: 
\begin{eqnarray}\label{e.riemann}
 n^{2H}
 \lp
\frac{1}{n}  \sum_{0\leq t_{k}<t} y_{t_{k}}  - \int_{0}^{t} y_{u}du\rp    \to   0 
\qquad \text{ as $n\to\infty$.}
\end{eqnarray}

%

\end{prop}
\begin{proof}
The proof is divided into several steps.

\noindent\emph{Step 1: A decomposition of the error of Riemann sum.} 
We  first have  
\begin{eqnarray}\label{e.ydu}
 \int_{0}^{t} y_{u}du -\frac{1}{n}  \sum_{0\leq t_{k}<t} y_{t_{k}}  &=& 
 \sum_{0\leq t_{k}<t}  \int_{t_{k}}^{t_{k+1}} \delta y_{t_{k}u} du
.
\end{eqnarray}
Substituting the expansion   $\delta y_{t_{k}u} = y_{t_{k}}'x^{1}_{t_{k}u}+ y_{t_{k}}''x^{2}_{t_{k}u}+r^{(0)}_{t_{k}u}$  into \eqref{e.ydu} we get the   expansion: 
\begin{eqnarray}\label{e.riemanni}
 \int_{0}^{t} y_{u}du-\frac{1}{n}  \sum_{0\leq t_{k}<t} y_{t_{k}}   &=&  I_{1}+I_{2}+I_{3},
\end{eqnarray}
where
\begin{eqnarray*}
I_{1} =  \sum_{0\leq t_{k}<t}  y_{t_{k}}' \int_{t_{k}}^{t_{k+1}} x^{1}_{t_{k}u} du , 
  \qquad
I_{2} =   \sum_{0\leq t_{k}<t} y_{t_{k}}''\int_{t_{k}}^{t_{k+1}} x^{2}_{t_{k}u} du, 
\qquad
I_{3} =  \sum_{0\leq t_{k}<t}\int_{t_{k}}^{t_{k+1}} r^{0}_{t_{k}u} du. 
\end{eqnarray*}
In the following we consider  the convergence of $I_{1}$, $I_{2}$ and $I_{3}$ which together will give the desired  convergence in \eqref{e.riemann}.

\noindent\emph{Step 2: Convergence of $I_{1}$ and $I_{3}$.} 
Since   $|r^{(0)}_{t_{k}u}|_{L_{1}}\lesssim n^{-3H}$ it  follows that 
 \begin{eqnarray}\label{e.i3c}
n^{2H}I_{3}\to 0 
\end{eqnarray}
in probability as $n\to\infty$. On the other hand, a direct application of  Lemma \ref{lem.yxdt} yields  the convergence
 \begin{eqnarray}\label{e.i1c}
n^{2H} I_{1} \to   -\frac{1}{4H+2}  \int_{s}^{t}y''_{u}du
\qquad \text{as $n\to\infty$}.
\end{eqnarray}


\noindent\emph{Step 3: Convergence of $I_{2}$.} 
 We consider the following decomposition of $I_{2}$:
 \begin{eqnarray*}
I_{2} &=&   \sum_{0\leq t_{k}<t} y_{t_{k}}''\int_{t_{k}}^{t_{k+1}} x^{2}_{t_{k}u} du = I_{21}
+I_{22},
\end{eqnarray*}
where
\begin{eqnarray}
I_{21}
&=& \sum_{0\leq t_{k}<t} y_{t_{k}}''\int_{t_{k}}^{t_{k+1}} (x^{2}_{t_{k}u} - \frac12 (u-t_{k})^{2H} ) du 
\label{e.i21}
\\
I_{22}
&=&\frac12 \sum_{0\leq t_{k}<t} y_{t_{k}}''\int_{t_{k}}^{t_{k+1}}   (u-t_{k})^{2H}   du . 
\notag
\end{eqnarray}
It is clear  that
\begin{eqnarray*}
I_{22} =\frac12 \sum_{0\leq t_{k}<t} y_{t_{k}}'' \cdot (1/n)^{2H+1}(2H+1)^{-1} .
\end{eqnarray*}
It follows that
\begin{eqnarray*}
n^{2H}I_{22} \to \frac{1}{4H+2}  \int_{0}^{t} y_{u}'' du
\qquad \text{in probability as $n\to\infty$.}
\end{eqnarray*}

We turn to the convergence of $I_{21}$. 
We first note that  a direct computation shows that 
\begin{eqnarray}
\Big|  \sum_{0\leq t_{k}<t}  \int_{t_{k}}^{t_{k+1}}  \Big( x^{2}_{t_{k}u} - \frac12 (u-t_{k})^{2H}  \Big) du \Big|_{L_{2}}^{2} &\lesssim& \sum_{0\leq t_{k},t_{k'}<t}  \int_{t_{k}}^{t_{k+1}}\int_{t_{k'}}^{t_{k'+1}}
n^{-4H}|\rho (k-k')|^{2}
du'du
\notag
\\
&\lesssim& n^{-4H-2} n(t-s) = n^{-4H-1}  (t-s) .
\label{e.i21a} 
\end{eqnarray}
Applying  Proposition \ref{prop.yhbd} to $I_{21}$ in \eqref{e.i21} with $\ell=1$ and $\be_{0}=1-H+\ep$ and invoking  the estimate \eqref{e.i21a} we obtain 
\begin{eqnarray*}
n^{2H}|I_{21}|_{L_{1}} \lesssim (t-s)^{1-H+\ep} (1/n)^{H-\ep}   
\end{eqnarray*}
for any $\ep>0$. In particular, we have  $n^{2H}|I_{21}|_{L_{1}} \to0$ as $n\to\infty$. 
Combining the convergence of $I_{21}$ and $I_{22}$ we obtain
\begin{eqnarray}\label{e.i2c}
n^{2H}I_{2}  \to   \frac{1}{4H+2}   \int_{0}^{t} y_{u}'' du
\qquad \text{in probability as $n\to\infty$.}
\end{eqnarray}

\noindent\emph{Step 4: Conclusion.} Substituting the convergences of $I_{i}$, $i=1,2,3$ in \eqref{e.i3c}, \eqref{e.i1c} and \eqref{e.i2c} into \eqref{e.riemanni} we obtain the   convergence \eqref{e.riemann}. 
\end{proof}

\begin{remark} 
The proof of  Proposition \ref{prop.riemann} 
suggests that  the exact rate of   convergence in \eqref{e.riemann}  is   $O(n^{-H-1/2})$ given that $y$ satisfies some regularity conditions. But the rate $o(n^{-2H})$ we have obtained  in  \eqref{e.riemann}   is sufficient for our purpose in this paper, and it requires a weaker condition of $y$. 
\end{remark}

In the next result we consider the   convergence rate of the Riemann sum under a weaker condition. We will also include the case when $H=1/2$.  
 

\begin{prop}\label{prop.riemann2}
Let $x$ be a one-dimensional fBm with Hurst parameter $H \leq 1/2$.  Let $(y,y')$ be a   rough path controlled by $(x, 2, H-\ep)$ in $L_{2}$ for some $\ep>0$.  
 Then there is a constant $K$ independent of $n$  such that: 
\begin{eqnarray}\label{e.riemann2}
\Big|\frac{1}{n}  \sum_{0\leq t_{k}<t} y_{t_{k}}  - \int_{0}^{t} y_{u}du\Big|_{L_{1}}
\leq Kn^{-2H+2\ep} 
\end{eqnarray}
for all $t\in[0,1]$.

\end{prop}
\begin{proof}
Because $y$ is controlled by $(x,2, H)$ we have  the relation $\delta y_{t_{k}u} = y_{t_{k}}'x^{1}_{t_{k}u}+  r^{(0)}_{t_{k}u}$. 
So, similar to \eqref{e.riemanni},    we have the decomposition 
\begin{eqnarray}\label{e.i1i2}
 \int_{0}^{t} y_{u}du-\frac{1}{n}  \sum_{0\leq t_{k}<t} y_{t_{k}}   &=&  I_{1}+I_{2}, 
\end{eqnarray}
where
\begin{eqnarray*}
I_{1} =  \sum_{0\leq t_{k}<t}  y_{t_{k}}' \int_{t_{k}}^{t_{k+1}} x^{1}_{t_{k}u} du , 
  \qquad
I_{2} =     \sum_{0\leq t_{k}<t}\int_{t_{k}}^{t_{k+1}} r^{0}_{t_{k}u} du. 
\end{eqnarray*}
It is readily checked that $|I_{2}|_{L_{1}}\lesssim n^{-2H+2\ep}$. Let $h^{n}$ be defined in \eqref{e.hn}.  Applying Proposition \ref{prop.yhbd} with $h = n^{2H-2\ep}h^{n}$,  $\ell=1$ and $\be_{0}=1-H+2\ep$ we obtain that
\begin{eqnarray*}
n^{2H-2\ep}|I_{1}|_{L_{1}}\lesssim 1.
\end{eqnarray*}
Combining the estimate of $I_{1}$ and $I_{2}$ in \eqref{e.i1i2} we obtain the desired estimate   \eqref{e.riemann2}.
\end{proof}



    \subsection{Weighted $p$-variations}
    
  In this subsection we consider limit theorems for weighted random sums of some  fBms    functionals. 
  For $p>-1$, we denote   
  \begin{eqnarray}\label{e.cp}
 c_{p} = \mE(|N|^{p}) = \frac{2^{p/2} }{\sqrt{\pi}} \Gamma \left(\frac{p+1}{2}\right).
\end{eqnarray}
It is easy to see that $c_{p+2}=(p+1)c_{p}$, and  when $p$ is an even integer we have $c_{p}=\mE(N^{p}) = (p-1)(p-3)\cdots 1$. 
We define the sign function: 
\begin{eqnarray}\label{e.sign}
\text{$\text{sign}(x)=1, -1, 0$ for $x>0$, $x<0$ and  $x=0$, respectively.}
\end{eqnarray}

\begin{lemma}\label{prop.sign}
Let $x$  be 
       a fBm with Hurst parameter $H<1/2$.   
       Let   $ (y ,y'  )$ be a   process   controlled by $(x, 2 , H )$ almost surely. 
   Take $p>1/2$ and  let 
    \begin{eqnarray}\label{e.xps}
f (x)=|x|^{p+1}\cdot \emph{sign}(x) , \qquad x\in\RR.
\end{eqnarray}
Then  we have the  following convergence in probability: 
\begin{eqnarray}\label{e.conv2}
n^{H-1}   \sum_{0\leq t_{k}<t} y_{t_{k}}     
  f(n^{H}  x^{1}_{t_{k}t_{k+1}})\to      -\frac{1}{2 } c_{p+2} \int_{0}^{t} y_{u}' du  
\qquad \text{as $n\to\infty$.}
\end{eqnarray}

\end{lemma}
\begin{proof}
We prove the  convergence \eqref{e.conv2} by applying   \cite[Theorem 4.14 (ii)]{LT20}. 
It is easy to see that  the function $f$ in \eqref{e.xps}  belongs to $L_{2}(\ga)$  with  Hermite rank $d=1$ as long as $p>-3/2$. Take $\ell=d+1=2$. By assumption $(y,y')$ is a   rough path controlled by $(x,\ell, H)$ almost surely. Furthermore, it is easy to see that $f\in W^{2,2}(\RR, \ga)$ when $p>1/2$.  In summary, we have shown that 
  the conditions in   \cite[Theorem 4.14 (ii)]{LT20} hold for the weighted sum in \eqref{e.conv2}.  
Since $f$ is an odd function it has the    decomposition $f(x) = \sum_{q=0}^{\infty} a_{2q+1}H_{2q+1}(x)$. We compute the first coefficient:     
\begin{eqnarray*}
a_{1} = \mE[ |N|^{p+1}\cdot \text{sign}(N)  N] = \mE[|N|^{p+2}] = c_{p+2}. 
\end{eqnarray*}
    Applying \cite[Theorem 4.14]{LT20}
    we thus obtain the convergence \eqref{e.conv2}.
\end{proof}

   Let     $f (x) = {|x|^{p}}  -c_{p}$, $x\in\RR$.  
   It is easily seen  that $f \in L_{2}(\ga)$ with  Hermite rank $d= 2$ when   $p>-\frac12$.  
Furthermore, we have  the decomposition $f(x) = \sum_{q=1}^{\infty}  {a}_{2q}H_{2q} (x) $, where the constants $a_{2q}$ are given by: 
  \begin{eqnarray}\label{e.a2q}
a_{2q} &=& \sum_{r=0}^{  q } \frac{(-1)^{r}}{2^{r} r! (2q-2r)!}  
c_{2q-2r+p} , \qquad q=1,2,\dots. 
\end{eqnarray} 
We also set the constant  $\si $: 
\begin{eqnarray}\label{e.pvar2}
 {\sigma}^{2} = \sum_{q=1}^{\infty} (2q)!  {a}^{2}_{2q} \sum_{k\in \ZZ}\rho(k)^{2q}   ,
\end{eqnarray}
where $\rho$ is defined in \eqref{e.rho}. 
Note that  when $H=1/2$ we have $\rho(0)=1$ and $\rho(k)=0$ for $k\neq 0$, and so \eqref{e.pvar2}   gives  $\si  = \|f\|_{L_{2}(\ga)}=(c_{2p}-c_{p}^{2})^{1/2}$. 

The next limit theorem result  is an application  of  \cite[Theorem 4.7 and Theorem 4.14]{LT20}. The proof is similar to Lemma \ref{prop.sign} and is omitted for sake of conciseness.   In the following $\xrightarrow{ \  stable~ f.d.d. \ }$ stands for the stable convergence of finite dimensional distributions. That is,  we say  $X^{n}_{t}\xrightarrow{ \  stable~ f.d.d. \ } X_{t}$, $t\in[0,1]$     
 if   the finite dimensional distribution of the process $X^{n}_{t}$, $t\in[0,1]$ converges stably  to that of the process $X_{t}$, $t\in [0,1]$ as $n\to\infty$.

%
   
\begin{prop}\label{prop.pvar}
Let $x$  be 
       a fBm with Hurst parameter $H\leq1/2$. 
       Let   $(y^{(0)},\dots, y^{(\ell-1)})$ be a   process controlled by $(x, \ell, H)$ almost surely for some $\ell\in\NN$. 
       Let  $a_{2q}$ and $\si$ be constants given in \eqref{e.a2q}-\eqref{e.pvar2}. 
Then: 

\noindent
(i)  For $\frac12\geq H >\frac14$,  $\ell=2$  
and $p\in( 3/2,\infty)$ we have the  convergence:
\begin{eqnarray}\label{e.pvar}
\frac{1}{\sqrt{n}} \sum_{0\leq t_{k}<t}  y_{t_{k}} ({| n^{H}  x^{1}_{t_{k}t_{k+1}}|^{p}}  - c_{p}) \xrightarrow{ \  stable~ f.d.d. \ }  {\sigma} \int_{0}^{t} y_{t} dW_{t}  
\qquad\text{for $t\in[0,1]$, }
\end{eqnarray}
where $W$ is a Wiener process independent of $x$. 

\noindent(ii) For $H = \frac14$,  $\ell=3$ and $p\in(7/2,\infty)\cup\{ 2\}$ we have the  convergence:
\begin{eqnarray*}
\frac{1}{\sqrt{n}} \sum_{0\leq t_{k}<t} y_{t_{k}} ({| n^{H}  x^{1}_{t_{k}t_{k+1}}|^{p}}  - c_{p}) \xrightarrow{ \  stable~ f.d.d. \ }   {\sigma} \int_{0}^{t} y_{u} dW_{u}+ \frac{ pc_{p}}{8}   \int_{0}^{t} y_{u}''du    
\qquad\text{for $t\in[0,1]$. } 
\end{eqnarray*}

\noindent(iii) For $H < \frac14$, $\ell=3$ and $p\in(7/2,\infty)\cup\{ 2\}$ we have the    convergence in probability:
\begin{eqnarray*}
n^{2H -1} \sum_{0\leq t_{k}<t} y_{t_{k}} ({| n^{H}  x^{1}_{t_{k}t_{k+1}}|^{p}}  - c_{p}) \xrightarrow{ \   \ }  \frac{ pc_{p}}{8}   \int_{0}^{t} y_{u}''du   
\qquad\text{for $t\in[0,1]$. }
\end{eqnarray*}

\end{prop}

\section{Limit theorem for $p$-variation of   processes controlled by fBm}\label{section.main}

In this section we consider the convergence of $p$-variation for processes controlled by fBm.  
Throughout the section we let  
$\phi(x) = |x|^{p} $, $ x\in\RR$    
for $p\geq 1$. 
 We first state   the following elementary result. 
 \begin{lemma}\label{lem.phi}
Denote by $\phi^{(j)}$   the $j$th derivative of $\phi$. 
For convenience we will also   write $\phi(x)=\phi^{(0)}(x)$, $\phi'(x)=\phi^{(1)}(x) $ and $\phi''(x)=\phi^{(2)}(x)$.   For   $j=0,1,\dots, [p]$ we  set
\begin{eqnarray}\label{e.fj}
\phi_{j}(x)= |x|^{p-j}\cdot \emph{sign}(x)^{j}
\qquad
\text{and}
\qquad
K_{j} =p \cdots (p-j+1)= \prod_{i=0}^{j-1} (p-i),
\end{eqnarray}
where recall that $\emph{sign}(x)$ is defined in \eqref{e.sign} and we use  the  convention that  $\prod_{i=0}^{-1} (p-i)=1$. For example, we have $K_{0}=1$, $K_{1}=p$, $K_{2}=p(p-1)$.   Then  

\noindent(i)  When $p$ is odd,    $\phi$ has derivative up to order $[p]-1$, and $\phi^{([p]-1)}$ is Lipschitz. When $p$ is  even, $\phi$ has derivative of all orders.   When $p$ is non-integer, $\phi$ has derivative up to order $[p]$. 

\noindent(ii) For $x\in\RR$ we have 
\begin{eqnarray}\label{e.fjp}
\phi^{(j)}(x) 
&=& K_{j}\cdot \phi_{j}(x), 
\qquad j=0,1,\dots, [p],  
\end{eqnarray} 
 with the exception that  $\phi^{(p)}(0)$ is undefined when $p$ is an odd number. 
In particular, 
when $p$ is an odd number we have
$\phi^{(j)}(x) = K_{j}x^{p-j}\emph{sign}(x)$, 
while when $p$ is   even   we get
$\phi^{(j)}(x) = K_{j}x^{p-j} $.
\end{lemma}

 Following is our  main result.  Recall that we define $\phi(x) = |x|^{p} $, $ x\in\RR$, and  for  a continuous  process $y$     the $p$-variation of $y$ over the time interval $[0,t]$ is defined as   
\begin{eqnarray*}
\sum_{0\leq t_{k}<t}  \phi(\delta y_{t_{k}t_{k+1}})=\sum_{0\leq t_{k}<t} |\delta y_{t_{k}t_{k+1}}|^{p}  , 
\end{eqnarray*}  
where $t_{k}=k/n$, $k=0,1,\dots,n$ is a  uniform partition of $[0,1]$. 
\begin{thm}\label{thm.main}
Let $x$  be 
       a fBm with Hurst parameter $H\leq 1/2$ and   $(y^{(0)},\dots, y^{(\ell-1)})$ be a   process controlled by $(x, \ell, H)$ almost surely (see Definition \ref{def.control}) for some $\ell\in\NN$. Recall that    $\phi'$ and $\phi''$  are derivatives of $\phi$ defined in   \eqref{e.fjp}, and   $c_{p}$ and $\si$ are     constants   given  in \eqref{e.cp} and \eqref{e.pvar2}, respectively.   
Let 
\begin{eqnarray*}
U^{n}_{t} = n^{pH-1}
\sum_{0\leq t_{k}<t} \phi(\delta y_{t_{k}t_{k+1}})  - c_{p}\int_{0}^{t} \phi(y'_{u}) du
 \qquad t\in[0,1].  
\end{eqnarray*}
 Then

\noindent (i) When $1/4< H\leq1/2$, $p\in[3,\infty)\cup\{2\}$  and $\ell\geq4$ we have the stable f.d.d.  convergence 
\begin{eqnarray}\label{e.uconv1}
\lp n^{1/2}U^{n}  , x   \rp \to (U,x),
\qquad \text{as $n\to\infty$, }
 \end{eqnarray}
 where
 \begin{eqnarray*}
U_{t} = \si \int_{0}^{t} \phi(y_{u}') dW_{u}\qquad t\in[0,1]\,, 
\end{eqnarray*}
and $W$ is a standard Brownian motion independent of $x$. 

\noindent (ii) When $H=1/4$, $p\in[5,\infty)\cup\{2,4\}$ and $\ell\geq 6$ we have the stable f.d.d.  convergence 
\begin{eqnarray}\label{e.uconv2}
\lp n^{1/2}U^{n}  , x   \rp \to (U,x),
\qquad \text{as $n\to\infty$, }
 \end{eqnarray}
 where
 \begin{eqnarray*}
U_{t} = \si \int_{0}^{t} \phi(y_{u}')dW_{u} 
- \frac{c_{p}}{8} \int_{0}^{t}  
    \phi''(y_{u}')(y_{u}'')^{2} du+ \frac{(p-2)c_{p}}{24} \int_{0}^{t}  \phi'(y_{u}') y_{u}''' 
    du . 
\end{eqnarray*}



 \noindent (iii) When $H<1/4$, $p\in[5,\infty)\cup\{2,4\}$  and $\ell\geq 6$ 
 we have the  convergence in probability 
\begin{eqnarray}\label{e.uconv3}
 n^{2H}U^{n}_{t}   \to U_{t}
\qquad
  \text{as $n\to\infty$  }
\end{eqnarray}
for  $t\in[0,1]$, where
 \begin{eqnarray*}
U_{t} = - \frac{c_{p}}{8} \int_{0}^{t}  
    \phi''(y_{u}')(y_{u}'')^{2} du+ \frac{(p-2)c_{p}}{24} \int_{0}^{t}  \phi'(y_{u}') y_{u}''' 
    du. 
\end{eqnarray*}

\end{thm}

  
\begin{proof}
 Take $
\ep>0$ sufficiently small.  Recall that $y^{(i)}$, $r^{(i)}$, $i=0,\dots,\ell-1$ and $G_{\bfy}=G_{\bfy,\ep}$ are defined in Definition \ref{def.control}.  Let $G_{x}=G_{x,\ep}$ be a finite random variable such that $|x^{1}_{st}|\leq G_{x} (t-s)^{H-\ep}$. 
  By localization (cf. \cite[Lemma 3.4.5]{JP}) we can and will assume that  there exists some constant $C_{0}>0$  such that  
\begin{eqnarray}\label{e.local}
\sum_{i=0}^{\ell-1}
\Big(
\sup_{t\in [0,1]} |y^{(i)}_{t}|+\sup_{s,t\in [0,1]} |r^{(i)}_{st}|
\Big)+G_{\bfy} +G_{x} 
<C_{0}
\qquad \text{almost surely.} 
\end{eqnarray}
Note that under this assumption it is clear that $(y^{(0)},\dots, y^{(\ell-1)})$ is  controlled by $(x, \ell, H-\ep)$ in $L_{p}$ for any $p>0$ (see Definition \ref{def.control}).

We divide the proof  into several steps. 

\noindent\emph{Step 1: Taylor's expansion of the function $\phi$.}
For convenience let us  denote   
\begin{eqnarray*}
q = \begin{cases}
 p&\text{when $p$ is an even number.}\\
  [p]-1&\text{otherwise.}
\end{cases}
\end{eqnarray*}
 Applying the    Taylor expansion to $\phi(\delta y_{t_{k}t_{k+1}}) $ at the value $y^{(1)}_{t_{k}}  \delta x_{t_{k}t_{k+1}}$ we get
\begin{eqnarray}\label{e.phi.i}
\phi(\delta y_{t_{k}t_{k+1}}) 
&=&I_{1}+I_{2} ,  
\end{eqnarray}
where
\begin{eqnarray}\label{e.i1}
I_{1}&=& \sum_{j=0}^{q} \frac{ \phi^{(j)}  (y^{(1)}_{t_{k}}  x^{1}_{t_{k}t_{k+1}})}{j!}   \cdot (\delta y_{t_{k}t_{k+1}} - y^{(1)}_{t_{k}}\delta x_{t_{k}t_{k+1}})^{j} 
\\
I_{2}&=& 
\frac{\phi^{(q+1)}(\xi_{k})}{(q+1) !}   \cdot (\delta y_{t_{k}t_{k+1}} - y^{(1)}_{t_{k}}\delta x_{t_{k}t_{k+1}})^{q+1} , 
\label{e.i2}
\end{eqnarray}
where $\xi_{k}$ is some value between $\delta y_{t_{k}t_{k+1}} $ and $y^{(1)}_{t_{k}}\delta x_{t_{k}t_{k+1}}$. 

\noindent\emph{Step 2: Estimate of $I_{2}$.}
  We first note that when $p$ is an even number $\phi^{(q+1)}(\xi_{k})=\phi^{(p+1)}(\xi_{k})= 0$, and so $I_{2}=0$. In the following we assume that $p$ is not   even and by definition of $q$ we have $q+1=[p] $.  
It is clear that 
\begin{eqnarray}\label{e.cbd}
|\xi_{k}|\leq |\delta y_{t_{k}t_{k+1}}| +  |y^{(1)}_{t_{k}}\delta x_{t_{k}t_{k+1}} | . 
\end{eqnarray}
The relation  \eqref{e.cbd} together with   the definition of $\phi^{(q+1)}$ in  \eqref{e.fj}  yields    
\begin{eqnarray}\label{e.phi.ck}
| \phi^{(q+1)}(\xi_{k}) |=| \phi^{([p])}(\xi_{k}) |\lesssim |\delta y_{t_{k}t_{k+1}}|^{p-[p]}+  |y^{(1)}_{t_{k}}\delta x_{t_{k}t_{k+1}} |^{p-[ p]}. 
\end{eqnarray}
Since $y$ is controlled by $x$,  Definition \ref{def.control} and the assumption \eqref{e.local}  gives 
\begin{eqnarray*}
|\delta y_{t_{k}t_{k+1}}|\leq G_{\bfy}(1/n)^{H-\ep}\leq C_{0} (1/n)^{H-\ep}.
\end{eqnarray*}
  Similarly, we have  $ |y^{(1)}_{t_{k}}\delta x_{t_{k}t_{k+1}} | \lesssim (1/n)^{H-\ep}$. Substituting these two estimates into \eqref{e.phi.ck}  we get 
\begin{eqnarray}\label{e.phi.ck.bd}
| \phi^{([p])}(\xi_{k}) | \lesssim (1/n)^{(p-[p])H-\ep}\wedge 1, 
\end{eqnarray}
where we added $\wedge 1$ to include the case when $p$ is odd. 

By Definition \ref{def.control} of controlled processes again  
we   have the estimate   $|\delta y_{t_{k}t_{k+1}} - y^{(1)}_{t_{k}}\delta x_{t_{k}t_{k+1}}|_{L_{2}}\lesssim (1/n)^{2H-\ep}$. Applying this estimate  and the estimate   \eqref{e.phi.ck.bd} to \eqref{e.i2}  we obtain  
\begin{eqnarray}\label{e.i2bd}
\Big|\sum_{0\leq t_{k}<t} I_{2}\Big |\leq 
\sum_{0\leq t_{k}<t}  |I_{2}  |
 \lesssim n\cdot (1/n)^{(p-[p])H-\ep}\cdot(1/n)^{2[p] H -\ep}  =(1/n)^{  pH+  [p] H-1-2\ep}     . 
\end{eqnarray}
It follows from \eqref{e.i2bd}   that  when $1/2\geq H>1/4$ and $p\geq 2$ we have 
\begin{eqnarray}\label{e.i2c1}
  n^{ pH-1/2} \sum_{0\leq t_{k}<t} I_{2}\to 0\qquad\text{in probability as $n\to\infty$ }
\end{eqnarray}
  and when $H\leq1/4$ and $p\geq 3$  we have  
\begin{eqnarray}\label{e.i2c2}
n^{ (p+2)H-1}\sum_{0\leq t_{k}<t} I_{2}\to 0\qquad\text{in probability as $n\to\infty$.}
\end{eqnarray}
  This shows that  $I_{2}$ does not have contribution in the limits  of $U^{n}$ in \eqref{e.uconv1}-\eqref{e.uconv3} under the given conditions in Theorem \ref{thm.main}. 

\noindent\emph{Step 3:    Decomposition of $I_{1}$.} 
Recall that $\phi^{(j)}(x)$ and $\phi_{j}(x)$ are defined in \eqref{e.fj}-\eqref{e.fjp}. It is clear that 
\begin{eqnarray}\label{e.fjab}
\phi^{(j)}(a\cdot b) = K_{j} 
\phi_{j}(a\cdot b) =K_{j}  \phi_{j}(a)\phi_{j}(b)
\qquad
\text{for any $a$ and $b\in \RR$}.
\end{eqnarray}
On the other hand,  we  rewrite the relation \eqref{e.y} as:     
 \begin{eqnarray}\label{e.y1}
\delta y_{t_{k}t_{k+1}} - y^{(1)}_{t_{k}}\delta x_{t_{k}t_{k+1}}=\sum_{i=2}^{\ell-1} y^{(i)}_{t_{k}} x^{i}_{t_{k}t_{k+1}} +r^{(0)}_{t_{k}t_{k+1}}.
\end{eqnarray}
 Substituting \eqref{e.fjab}-\eqref{e.y1} into  \eqref{e.i1} we obtain  
\begin{eqnarray}\label{e.i1i}
I_{1} =  \sum_{j=0}^{q} \frac{K_{j}\phi_{j} (y^{(1)}_{t_{k}} )\phi_{j}(x^{1}_{t_{k}t_{k+1}})}{j!}   \cdot \lp \sum_{i=2}^{\ell-1} y^{(i)}_{t_{k}} x^{i}_{t_{k}t_{k+1}} +r^{(0)}_{t_{k}t_{k+1}}\rp^{j} . 
\end{eqnarray}
In the following we consider two different decompositions of $I_{1}$  in  \eqref{e.i1i}
 according to the value of~$H$.

When $H\leq 1/4$ we consider the   decomposition:
\begin{eqnarray}\label{e.i1d}
I_{1} = J_{1}+ J_{2}+J_{3}+J_{4}+J_{5}+J_{6},  
\end{eqnarray}
where
\begin{eqnarray}
J_{1} &=&   {K_{0}\phi_{0} (y^{(1)}_{t_{k}} )\phi_{0}(x^{1}_{t_{k}t_{k+1}})}   = |y^{(1)}_{t_{k}} |^{p} \cdot |x^{1}_{t_{k}t_{k+1}}|^{p}
\notag
\\
J_{2} &=&  {K_{1}\phi_{1} (y^{(1)}_{t_{k}} )\phi_{1}(x^{1}_{t_{k}t_{k+1}})}    \cdot   y^{(2)}_{t_{k}} x^{2}_{t_{k}t_{k+1}}  
\label{e.j2}
\\
J_{3} &=&  {K_{1}\phi_{1} (y^{(1)}_{t_{k}} )\phi_{1}(x^{1}_{t_{k}t_{k+1}})}    \cdot   y^{(3)}_{t_{k}} x^{3}_{t_{k}t_{k+1}}
\notag
\\
J_{4}&=&   \frac{K_{2}\phi_{2} (y^{(1)}_{t_{k}} )\phi_{2}(x^{1}_{t_{k}t_{k+1}})}{2!}   \cdot \lp  y^{(2)}_{t_{k}} x^{2}_{t_{k}t_{k+1}} \rp^{2} 
\notag
\\
J_{5} &=&   \sum_{j=0}^{q} \frac{K_{j}\phi_{j} (y^{(1)}_{t_{k}} )\phi_{j}(x^{1}_{t_{k}t_{k+1}})}{j!}   \cdot \lp \sum_{i=2}^{\ell-1} y^{(i)}_{t_{k}} x^{i}_{t_{k}t_{k+1}}  \rp^{j} - \sum_{e=1}^{4}J_{e} . 
\label{e.j5}
\\
J_{6} &=& I_{1} -  \sum_{j=0}^{q} \frac{K_{j}\phi_{j} (y^{(1)}_{t_{k}} )\phi_{j}(x^{1}_{t_{k}t_{k+1}})}{j!}   \cdot \lp \sum_{i=2}^{\ell-1} y^{(i)}_{t_{k}} x^{i}_{t_{k}t_{k+1}}  \rp^{j} . 
\label{e.j6}
\end{eqnarray}
When $H>1/4$ we consider the decomposition 
\begin{eqnarray}\label{e.i1.d}
I_{1}=J_{1}+J_{6}+(I_{1}-J_{1}-J_{6}).
\end{eqnarray}

\noindent\emph{Step 4:    Estimate of $J_{6}$.} 
Recall that $J_{6}$ is defined in \eqref{e.j6}.  Note that $J_{6}$ consists of the terms in \eqref{e.i1i} which contain $r^{(0)}_{t_{k}t_{k+1}}$. Similar to the estimate in \eqref{e.phi.ck.bd}, invoking the definition of controlled processes (see Definition \ref{def.control}) and the assumption \eqref{e.local}   we have 
\begin{eqnarray*}
|r^{(0)}_{t_{k}t_{k+1}}| \lesssim (1/n)^{\ell H -\ep} \qquad\text{and} 
\qquad
| y^{(i)}_{t_{k}} x^{i}_{t_{k}t_{k+1}} |\lesssim (1/n)^{2H-\ep },\quad i=2,\dots, \ell-1.  
\end{eqnarray*}
It follows that 
\begin{eqnarray}\label{e.i6.r}
\left| 
\lp \sum_{i=2}^{\ell-1} y^{(i)}_{t_{k}} x^{i}_{t_{k}t_{k+1}} +r^{(0)}_{t_{k}t_{k+1}}\rp^{j} 
-
\lp \sum_{i=2}^{\ell-1} y^{(i)}_{t_{k}} x^{i}_{t_{k}t_{k+1}} \rp^{j} 
\right|\lesssim   (1/n)^{\ell H -\ep}
 (1/n)^{ 2H \cdot (j-1) -\ep   }
  . 
\end{eqnarray}
On the other hand, by the definition of $\phi_{j}$ in \eqref{e.fj} we have the estimate 
\begin{eqnarray}\label{e.i6.phi}
\left|  \frac{K_{j}\phi_{j} (y^{(1)}_{t_{k}} )\phi_{j}(x^{1}_{t_{k}t_{k+1}})}{j!}  \right|\lesssim (1/n)^{(p-j)H-\ep}.
\end{eqnarray}
Substituting the two estimates \eqref{e.i6.r}-\eqref{e.i6.phi} into \eqref{e.j6} we obtain  
\begin{eqnarray*}
|   J_{6}| \lesssim \sum_{j=1}^{q}  (1/n)^{( \ell +2(j-1) +(p-j) )H -\ep} \lesssim   (1/n)^{( \ell -1 +p )H  -\ep}. 
\end{eqnarray*}
It follows that 
\begin{eqnarray*}
| \sum_{0\leq t_{k}<t}  J_{6}| \leq  \sum_{0\leq t_{k}<t} | J_{6}|  
 \lesssim   (1/n)^{( \ell -1 +p )H -1-\ep}. 
\end{eqnarray*}
It is readily checked  that 
\begin{eqnarray}\label{e.j6c1}
   n^{pH-1/2} \sum_{0\leq t_{k}<t}  J_{6}\to 0 \qquad \text{when   $\ell \geq 3$ and $H>1/4$}
\end{eqnarray}
   and
    \begin{eqnarray}\label{e.j6c2}
  n^{(p+2)H-1 } \sum_{0\leq t_{k}<t}  J_{6}\to 0  \qquad\text{when $\ell \geq 4$ and $H\leq1/4$}. 
\end{eqnarray}
 Note that this shows that  $J_{6}$   does not have contribution in any of the limits  of $U^{n}$ in \eqref{e.uconv1}-\eqref{e.uconv3}. 

\noindent\emph{Step 5:    Convergence of $J_{1}$.}
We first consider the case when $1/2\geq H>1/4$.  
According to  Proposition \ref{prop.pvar}  
given that    $p>3/2$ and that $|y'|^{p}$ is controlled by $(x, 2, H)$ 
  we have the stable f.d.d. convergence: 
 \begin{eqnarray}\label{e.j1c}
\frac{1}{\sqrt{n}} \sum_{0\leq t_{k}<t} ( n^{pH} J_{1}   - |y_{t_{k}}'|^{p} c_{p}) \xrightarrow{ \  \ }  {\sigma} \int_{0}^{t} |y_{t}'|^{p} dW_{t}
\qquad
\text{ as $n\to\infty$}.
\end{eqnarray}
Note that by Lemma \ref{lem.control} (ii) and Lemma \ref{lem.phi} (i) for $|y'|^{p}$ to be controlled by $(x, 2, H)$ it requires   $p\geq 2$ and   that  $y$ is controlled by $(x,\ell,H)$ for $\ell \geq 3$. 

On the other hand, given that 
 $|y'|^{p}$ is controlled by $(x, 2, H)$   Proposition \ref{prop.riemann2} implies that 
  \begin{eqnarray}\label{e.j1c2}
\frac{1}{\sqrt{n}} \sum_{0\leq t_{k}<t}  |y_{t_{k}}'|^{p} c_{p}   - \sqrt{n}\cdot c_{p}  \int_{0}^{t}|y_{u}'|^{p} du \to 0
\qquad
\text{ as $n\to\infty$}.
\end{eqnarray}
  Combining \eqref{e.j1c2} with the convergence in \eqref{e.j1c}  we obtain   the stable f.d.d. convergence 
\begin{eqnarray}\label{e.j1ci}
&&
\sqrt{n} \lp  \sum_{0\leq t_{k}<t}  n^{pH-1} J_{1}   -  c_{p}  \int_{0}^{t}|y_{u}'|^{p} du \rp
\xrightarrow{ \  \ }  {\sigma} \int_{0}^{t} |y_{u}'|^{p} dW_{u}
\qquad
\text{ as $n\to\infty$}.
\end{eqnarray}

We turn to the case when $H=1/4$.  According to  Proposition \ref{prop.pvar} given that
   $|y'|^{p}$ is controlled by $(x, 3, H)$, $p\in(7/2,\infty)\cup\{ 2\}$   we have the stable f.d.d. convergence: 
\begin{eqnarray*}
\frac{1}{\sqrt{n}} \sum_{0\leq t_{k}<t} ( n^{pH} J_{1}   - |y_{t_{k}}'|^{p} c_{p}) \xrightarrow{ \   \ }  {\sigma} \int_{0}^{t} |y_{u}'|^{p} dW_{u}+\frac{pc_{p}}{8}\int_{0}^{t} (|y_{u}'|^{p})'' du
\qquad
\text{as $n\to\infty$},
\end{eqnarray*}
where $(|y_{u}'|^{p})'' = (\phi(y'_{u}))'' = \phi''(y'_{u})(y''_{u})^{2}+\phi'(y'_{u})y'''_{u}$. 
By Lemma \ref{lem.control} (ii)  and Lemma \ref{lem.phi} (i) again this requires  $p\in[ 3,\infty)\cup\{2\}$   and $\ell\geq  4$.
Similar to   \eqref{e.j1ci}, we can apply Proposition \ref{prop.riemann} to  obtain  the stable f.d.d. convergence: 
\begin{eqnarray}\label{e.j1cii}
\sqrt{n} \lp  \sum_{0\leq t_{k}<t}  n^{pH-1} J_{1}   -  c_{p}  \int_{0}^{t}|y_{u}'|^{p} du \rp
  \xrightarrow{ \   \ }  {\sigma} \int_{0}^{t} |y_{u}'|^{p} dW_{u}+\frac{pc_{p}}{8}\int_{0}^{t} (|y_{u}'|^{p})'' du\,.
\end{eqnarray}

When $H<1/4$, given that $p\in(7/2,\infty)\cup\{ 2\}$ and $|y'|^{p}$ is controlled by $(x, 3, H)$ we have the convergence in probability:
\begin{eqnarray}\label{e.j1ciii}
n^{2H-1} \sum_{0\leq t_{k}<t} ( n^{pH} J_{1}   - |y_{t_{k}}'|^{p} c_{p}) \rightarrow \frac{pc_{p}}{8}\int_{0}^{t} (|y_{u}'|^{p})'' du\,.
\end{eqnarray}
Then it follows from  Proposition \ref{prop.riemann} again that   
\begin{eqnarray*}
n^{2H} \lp  \sum_{0\leq t_{k}<t}  n^{pH-1} J_{1}   -  c_{p}  \int_{0}^{t}|y_{u}'|^{p} du \rp
  \xrightarrow{} \frac{pc_{p}}{8}\int_{0}^{t} (|y_{u}'|^{p})'' du
  \qquad \text{in probability}.
\end{eqnarray*}


\noindent\emph{Step 6: Proof of \eqref{e.uconv1} and the convergence of   $(I_{1}-J_{1}-J_{6})$.}
 In this step we   show the convergence:   
 \begin{eqnarray}\label{e.i1rconv}
n^{pH-1/2}\sum_{0\leq t_{k}<t}(I_{1}-J_{1}-J_{6}) \to0.
\end{eqnarray}
 Combining \eqref{e.i1rconv}   with the convergences of $I_{2}$, $J_{6}$ and $J_{1}$ respectively in \eqref{e.i2c1}, \eqref{e.j6c1} and \eqref{e.j1ci}, and invoking the relations 
 \eqref{e.phi.i}  and \eqref{e.i1.d}   we then 
 obtain the    convergence in \eqref{e.uconv1}.

We first note that by the definition of $I_{1}$, $J_{1}$ and $J_{6} $ we have  
\begin{eqnarray}\label{e.i1r2}
\sum_{0\leq t_{k}<t} (I_{1}-J_{1}-J_{6} ) =
\sum_{0\leq t_{k}<t} 
   \sum_{j=1}^{q} \frac{K_{j}\phi_{j} (y^{(1)}_{t_{k}} )\phi_{j}(x^{1}_{t_{k}t_{k+1}})}{j!}   \cdot \lp \sum_{i=2}^{\ell-1} y^{(i)}_{t_{k}} x^{i}_{t_{k}t_{k+1}}  \rp^{j}  . 
\end{eqnarray}
It is easy to see that \eqref{e.i1r2} consists of weighted sums of the form $
\cj_{0}^{t} (z, h^{n}) $ for 
\begin{eqnarray}\label{e.i1rz}
z =  \frac{K_{j}\phi_{j} (y^{(1)}_{t_{k}} ) }{j!}   \cdot   \sum_{\substack{2\leq i_{1},\dots,i_{j}\leq\ell-1\\ i_{1}+\dots+i_{j}=r}} y^{(i_{1})}_{t_{k}}    \cdots y^{(i_{j})}_{t_{k}}      
\end{eqnarray}
and
\begin{eqnarray}\label{e.i1r}
h^{n}_{st}= \sum_{s\leq t_{k}<t} \phi_{j}(x^{1}_{t_{k}t_{k+1}}) \cdot (x^{1}_{t_{k}t_{k+1}})^{r}
= \sum_{s\leq t_{k}<t}   (x^{1}_{t_{k}t_{k+1}})^{r}|x^{1}_{t_{k}t_{k+1}}|^{p-j} \text{sign}(x^{1}_{t_{k}t_{k+1}})^{j}
\notag
\\
= \sum_{s\leq t_{k}<t}    |x^{1}_{t_{k}t_{k+1}}|^{p-j+r} \text{sign}(x^{1}_{t_{k}t_{k+1}})^{j+r}
. 
\end{eqnarray}
 for  $j=1,\dots, q$ and $r\geq 2j$. When $ p-j+r \geq p+2 $ we can bound $|h^{n}_{t_{k}t_{k+1}}|\lesssim (1/n)^{(p+2)H-\ep}$ and thus we have the convergence
 \begin{eqnarray}\label{e.i1rc}
n^{pH-1/2}\cj_{0}^{t}(z, h^{n}) \to 0 \qquad\text{in probability as $n\to\infty$.}
\end{eqnarray}

 In the following we consider the estimate of \eqref{e.i1r} when $ p-j+r < p+2 $, that is when $r-j<2$. Note that this implies    $r-j= 1$ and thus $j=1$ and $r=2$. So we have 
 \begin{eqnarray*}
\cj_{0}^{t}(z, h^{n}) = \sum_{0\leq t_{k}<t} K_{1}\phi_{1}(y'_{t_{k}}) \phi_{1}(x^{1}_{t_{k}t_{k+1}}) y''_{t_{k}}x^{2}_{t_{k}t_{k+1}}. 
\end{eqnarray*}
 Let $f(x)=\phi_{1}(x) x^{2} $, $x\in\RR$. It is clear that $f$ has Hermite rank $d=1$. 
 According to Proposition \ref{prop.yhbd2} (ii), given that  
 (a) $ f \in W^{2,2}(\RR,\ga)$ and (b) $z=\phi_{1}(y')y''$ defined in \eqref{e.i1rz} is controlled by $(x,2,H)$ we have  
 $\cj(z, h^{n})\sim (1/n)^{ H-1+ H(p+1) }$, which   implies that  the convergence \eqref{e.i1rc}   holds. By Lemma \ref{lem.control}(ii)    and Lemma \ref{lem.phi}(i) Condition (a) requires either $2(p+1-2)>-1$ or that $p$ is even, which means we must have  $p>1/2$.  By Lemma \ref{lem.control}(ii)    and Lemma \ref{lem.phi}(i)   Condition (b) holds when $p\in[ 3,\infty)\cup\{2\}$ and $\ell\geq4$.

 In summary, we have shown that \eqref{e.i1rc} holds for all $j$ and $r$. Invoking the decomposition of $I_{1}-J_{1}-J_{6}$ in \eqref{e.i1r2}-\eqref{e.i1r},    we conclude \eqref{e.i1rconv} and thus \eqref{e.uconv1}.

\noindent\emph{Step 7:    Convergence of $J_{2}$ when $H\leq 1/4$.}
Recall the definition of $\phi_{j}$ and $J_{2}$ in \eqref{e.fj} and \eqref{e.j2}, respectively. So we have
\begin{eqnarray*}
\sum_{0\leq t_{k}<t} J_{2} 
=\frac{p}{2} \sum_{0\leq t_{k}<t}  \phi_{1} (y'_{t_{k}} )y_{t_{k}}''  |x^{1}_{t_{k}t_{k+1}}|^{p+1} \text{sign}(x^{1}_{t_{k}t_{k+1}}).
\end{eqnarray*}
Suppose   that $\phi_{1} (y'_{t } )y_{t }'' $  is controlled by $(x,2,H)$. According to   Lemma \ref{lem.control}(ii) this requires $p\in[ 3,\infty)\cup\{2\}$ and $\ell\geq 4$, and in this case   we have 
\begin{eqnarray*}
(\phi_{1} (y'_{t } )y_{t }'')' =  \phi_{1}'(y_{t}')(y_{t}'')^{2} +\phi_{1}(y_{t}') y_{t}''' .
\end{eqnarray*}
 Applying   Lemma \ref{prop.sign} we obtain  the convergence in probability 
\begin{eqnarray}\label{e.j2c}
n^{(p+2)H-1} \sum_{0\leq t_{k}<t} J_{2}  \to \frac{p}{2}(-\frac{1}{2 } c_{p+2}) \int_{0}^{t} ( \phi_{1}'(y_{u}')(y_{u}'')^{2} +\phi_{1}(y_{u}') y_{u}''' ) du
\notag
\\
= -\frac{1}{4}(p+1)c_{p} \int_{0}^{t} ( \phi''(y_{u}')(y_{u}'')^{2} +\phi'(y_{u}') y_{u}''' ) du.
\end{eqnarray}

\noindent\emph{Step 8:    Convergence of $J_{3}$ when $H\leq 1/4$.}
We first rewrite $J_{3}$ as
\begin{eqnarray*}
J_{3} = {K_{1}\phi_{1} (y^{(1)}_{t_{k}} )\phi_{1}(x^{1}_{t_{k}t_{k+1}})}    \cdot   y^{(3)}_{t_{k}} x^{3}_{t_{k}t_{k+1}}
=\frac{p}{6}  \phi_{1} (y'_{t_{k}} )  y'''_{t_{k}}  |x^{1}_{t_{k}t_{k+1}}|^{p+2} .
\end{eqnarray*}
It is easy to see that we have the bound 
$
|\sum_{0\leq t_{k}<t}J_{3}|_{L_{p}}\lesssim (1/n)^{(p+2)H} . 
$ In the following we show that $\sum_{0\leq t_{k}<t}J_{3}$ is also convergence under proper conditions of $p$ and $\ell$. 

We consider the following decomposition 
\begin{eqnarray}\label{e.j3d}
J_{3} &=& \frac{p}{6}  \phi_{1} (y'_{t_{k}} )  y'''_{t_{k}} 
\lp
 |x^{1}_{t_{k}t_{k+1}}|^{p+2}-c_{p+2}(1/n)^{(p+2)H}\rp+ \frac{p}{6}c_{p+2}  \phi_{1} (y'_{t_{k}} )  y'''_{t_{k}} 
 (1/n)^{(p+2)H}
 \notag
 \\
 &=:& J_{31}+J_{32}. 
\end{eqnarray}

Applying Proposition \ref{prop.yhbd2} (ii)  to $J_{31}$ with  $d=2$ we obtain that $n^{(p+2)H-1}\sum_{0\leq t_{k}<t}J_{31}\to 0$ in probability. Note that the application of Proposition \ref{prop.yhbd2} (ii)  requires $p\in[ 4,\infty)\cup\{2\}$ and $\ell\geq 6$. On the other hand, by continuity of $\phi_{1}(y')y'''$ we have the convergence 
  $n^{(p+2)H-1} \sum_{0\leq t_{k}<t} J_{32}  \to \frac{p}{6} c_{p+2}   \int_{0}^{t} \phi_{1} (y'_{u} )  y'''_{u}du $. Substituting these two convergence into \eqref{e.j3d} we obtain 
\begin{eqnarray}\label{e.j3c}
n^{(p+2)H-1} \sum_{0\leq t_{k}<t} J_{3}  \to \frac{p}{6} c_{p+2}   \int_{0}^{t} \phi_{1} (y'_{u} )  y'''_{u}du 
\notag
\\
= \frac{1}{6}(p+1) c_{p}   \int_{0}^{t} \phi' (y'_{u} )  y'''_{u}du . 
\end{eqnarray}

\noindent\emph{Step 9:    Convergence of $J_{4}$ when $H\leq 1/4$.}
We first rewrite $J_{4}$ as
\begin{eqnarray*}
J_{4} = \frac{K_{2}\phi_{2} (y^{(1)}_{t_{k}} )\phi_{2}(x^{1}_{t_{k}t_{k+1}})}{2!}   \cdot \lp  y^{(2)}_{t_{k}} x^{2}_{t_{k}t_{k+1}} \rp^{2} 
\\
=\frac{p(p-1)}{8}   \phi_{2} (y'_{t_{k}} )  (y''_{t_{k}})^{2}  \cdot |x^{1}_{t_{k}t_{k+1}}|^{p+2} . 
\end{eqnarray*}
Similar to  $J_{3}$  by   applying  Proposition \ref{prop.pvar} (ii)-(iii) 
   with $d=2$ we obtain the convergence 
\begin{eqnarray}\label{e.j4c}
n^{(p+2)H-1} \sum_{0\leq t_{k}<t} J_{4}  \to \frac{p(p-1)}{8} c_{p+2}   \int_{0}^{t} \phi_{2} (y'_{u} )  (y''_{u})^{2}du
\notag 
\\
=\frac{1}{8}(p+1)c_{p} \int_{0}^{t} \phi''(y'_{u})(y''_{u})^{2}du. 
\end{eqnarray}
Note that the application of Proposition \ref{prop.pvar} (ii)-(iii) requires $p\in[ 5,\infty)\cup\{2,4\}$ and $\ell\geq 5$. 

 \noindent\emph{Step 10:    Convergence of $J_{5}$.}
Recall that $J_{5}$ is defined in \eqref{e.j5}. It is easy to see that we have 
\begin{eqnarray*}
J_{5}&=&  
\frac{K_{1}\phi_{1} (y^{(1)}_{t_{k}} )\phi_1(x^{1}_{t_{k}t_{k+1}})}{1!}   \cdot \lp \sum_{i=4}^{\ell-1} y^{(i)}_{t_{k}} x^{i}_{t_{k}t_{k+1}}  \rp 
\\
&&+
\frac{K_{2}\phi_2 (y^{(1)}_{t_{k}} )\phi_{2}(x^{1}_{t_{k}t_{k+1}})}{2!}   \cdot \lp \sum_{i=3}^{\ell-1} y^{(i)}_{t_{k}} x^{i}_{t_{k}t_{k+1}}  \rp^{2}
\\
&&+\sum_{j=3}^{q} \frac{K_{j}\phi_{j} (y^{(1)}_{t_{k}} )\phi_{j}(x^{1}_{t_{k}t_{k+1}})}{j!}   \cdot \lp \sum_{i=2}^{\ell-1} y^{(i)}_{t_{k}} x^{i}_{t_{k}t_{k+1}}  \rp^{j} 
\\
&=&J_{51}+J_{52}+J_{53},  
\end{eqnarray*}
where we use the convention that $\sum_{j=3}^{q}=0$ when $q<3$ and that $\sum_{i=4}^{\ell-1}y^{(i)}_{t_{k}} x^{i}_{t_{k}t_{k+1}} =0$ when $ \ell-1<4$. 
In the following we bound each $J_{5i}$, $i=1,2,3$.

For $J_{51}$   a direct estimate shows that  
\begin{eqnarray*}
|J_{51}|\lesssim 
 |\phi_{1} (y^{(1)}_{t_{k}} )|\cdot |\phi_1(x^{1}_{t_{k}t_{k+1}})| \cdot  \sum_{i=4}^{\ell-1} |y^{(i)}_{t_{k}}| \cdot|x^{i}_{t_{k}t_{k+1}}|    
 \lesssim  (1/n)^{(p-1)H+4H} =(1/n)^{(p+3)H }.  
\end{eqnarray*}
So we have $n^{(p+2)H-1}\sum_{0\leq t_{k}<t} J_{51}\to 0$. 
Similarly,  we can bound $J_{52}$ and $J_{53}$ by 
\begin{eqnarray*}
|J_{52}|\lesssim (1/n)^{(p-2)H+6H}
\qquad
\text{and}
\qquad
|J_{53}|\lesssim \sum_{j=3}^{\lfloor p\rfloor -1} (1/n)^{(p-j)H+2jH}\lesssim  (1/n)^{(p+3)H }.
\end{eqnarray*}
We conclude that  $n^{(p+2)H-1}\sum_{0\leq t_{k}<t} (J_{52}+J_{53})\to 0$ in probability as $n\to\infty$. 
We conclude that the convergence in probability 
\begin{eqnarray}\label{e.j5c}
n^{(p+2)H-1}\sum_{0\leq t_{k}<t} J_{5}\to 0
\qquad\text{as $n\to\infty$}.
\end{eqnarray}

    \noindent\emph{Step 11:    Conclusion.} 
         In Step 5 we have shown that  the convergence in \eqref{e.uconv1} holds. 
  Putting together the convergences \eqref{e.i2c2}, \eqref{e.j6c2}, \eqref{e.j1cii}, \eqref{e.j2c}, \eqref{e.j3c}, \eqref{e.j4c}, \eqref{e.j5c} for $I_{2}$, $J_{i}$, $i=1,\dots, 6$ and invoking Lemma \ref{lem.stable}, and then taking into account 
      the decompositions \eqref{e.phi.i} and \eqref{e.i1d}, 
      we obtain  the   convergence in \eqref{e.uconv2}.   
Finally, replacing \eqref{e.j1cii} by   \eqref{e.j1ciii}  in the argument  we obtain the   convergence  \eqref{e.uconv3}.  
\end{proof}

\bibliographystyle{abbrv}
 \bibliography{p-variation.bib}
 

\end{document}